\documentclass[a4paper, reqno]{amsart}
\usepackage{hyperref,url}
\usepackage{a4wide}
\usepackage{amsthm, amssymb, amsmath, latexsym, upgreek,accents}
\usepackage{amsfonts, mathrsfs, color}
\usepackage{graphicx,color}

\usepackage{tikz}

\usepackage{booktabs}
\usepackage{textcomp}
\usepackage{subfigure}

 \theoremstyle{plain}
 \begingroup
 \newtheorem{theo}{Theorem}[section]
 \newtheorem{lemma}[theo]{Lemma}
 \newtheorem{propo}[theo]{Proposition}
 
 \newtheorem{coro}[theo]{Corollary}
 \endgroup

 \theoremstyle{definition}
 \begingroup
 \newtheorem{defi}[theo]{Definition}
 
 \endgroup

 \theoremstyle{remark}
 \begingroup
 \newtheorem{rem}[theo]{Remark}
 \endgroup

 \numberwithin{equation}{section}

\mathchardef\emptyset="001F

\newcommand{\ds}{\displaystyle}

\newcommand{\dx}{\,dx}

\newcommand{\ie}{{; \it i.e., }}

\newcommand{\Am}{\mathcal{A}}

\let\e= \varepsilon
\newcommand{\R}{{\mathbb R}}
\newcommand{\Z}{{\mathbb Z}}

\newcommand{\N}{{\mathbb N}}
\newcommand{\Sph}{{\mathbb S}}

\newcommand{\Om}{\Omega}
\def\avint{\mathop{\,\rlap{--}\hspace{-.15cm}\int}\nolimits}

\newcommand{\mres}{\mathbin{\vrule height 1.6ex depth 0pt width 0.13ex\vrule height 0.13ex depth 0pt width 1.3ex}}

\DeclareMathOperator\supp{supp}

\title[Homogenization of high-contrast 
Mumford-Shah energies]{Homogenization of high-contrast \\
Mumford-Shah energies}

\author[X. Pellet]{Xavier Pellet}
\author[L. Scardia]{Lucia Scardia}
\address[X. Pellet and L. Scardia]{Department of Mathematical Sciences, University of Bath, Bath, United Kingdom}
\email[]{X.P.J.Pellet@bath.ac.uk}
\email[]{L.Scardia@bath.ac.uk}
\author[C. I. Zeppieri]{Caterina Ida Zeppieri}
\address[C. I. Zeppieri]{Angewandte Mathematik, WWU M\"unster, Germany}
\email[]{caterina.zeppieri@uni-muenster.de}
\begin{document}

\begin{abstract}
We prove a homogenization result for Mumford-Shah-type energies associated to a brittle composite material with weak inclusions distributed periodically at a scale $\e>0$. The matrix and the inclusions in the material have the same elastic moduli but very different toughness moduli, with the ratio of the toughness modulus in the matrix and in the inclusions being $1/\beta_\e$, with $\beta_{\varepsilon}>0$ small. We show that the high-contrast behaviour of the composite leads to the emergence of interesting effects in the limit: The volume and surface energy densities interact by $\Gamma$-convergence, and the limit volume energy is not a quadratic form in the critical scaling $\beta_\e = \e$, unlike the $\e$-energies, and unlike the extremal limit cases.


\end{abstract}

\maketitle

{\small

\medskip

}

\bigskip

\section{Introduction}

We study the homogenization of a family of Mumford-Shah-type free-discontinuity functionals representing the (linearly) elastic energy of a \textit{high-contrast} composite material constituted by a brittle matrix with weak inclusions distributed periodically. 
Our analysis is restricted to the case of an anti-plane shear, 
namely to scalar displacements $u:\Omega \subset \mathbb{R}^{n} \rightarrow \mathbb{R}$, where $\Omega\subset \mathbb{R}^n$ is a bounded and open set with Lipschitz boundary representing the cross-section of the reference configuration $\Omega \times \mathbb{R}$. The energy we consider is
		\begin{equation}\label{Fe}
			{F}_{\varepsilon}(u) 
			= \int_{\Omega}|\nabla u|^{2} dx 
			+ \mathcal{H}^{n-1}(S_{u}\cap\varepsilon P)
			+ \beta_{\varepsilon}\mathcal{H}^{n-1}(S_{u}\cap(\Omega\setminus\varepsilon P)),
		\end{equation}
where $\varepsilon>0$ is the ratio between the size of the microstructure and the observable length scale, and $\varepsilon P$ is the $\varepsilon$-scaled copy of a connected, $Q$-periodic, open set $P\subset\mathbb{R}^{n}$, 
with $Q=(-1/2,1/2)^{n}$, which models the matrix of the composite. The displacement $u$ is in the class $SBV^2(\Omega)$ of special functions of bounded variation, $\nabla u$ denotes its approximate gradient and $S_{u}$ its discontinuity set. The volume term in ${F}_{\varepsilon}$ represents the linearly elastic energy of the body, and the surface integral describes the energy needed to open a crack in the material. Note that the matrix and the inclusions have the same elastic moduli (normalised to $1$), but very different toughness moduli: the toughness modulus is $1$ in the matrix and $\beta_{\varepsilon}>0$ in the inclusions, with $\beta_{\varepsilon}\rightarrow 0$ as $\varepsilon\to 0$. This is why we call the brittle composite high-contrast.
\smallskip

The literature on high-contrast materials and on the derivation of their effective behaviour by homogenization is vast. In the classical Sobolev case, it is well known that interesting effects appear in the limit when the volume energy density does not satisfy uniform lower bounds (see, e.g., \cite{BKS, CC12, CC15}, and \cite{BrCPP} for the case of discrete energies). It is then natural to try to extend this analysis to the case of free-discontinuity functionals, and there has been a recent effort in this direction. Note that for free-discontinuity functionals the high contrast can be in the volume part of the energy \cite{Bar18,BLZ}, in the surface part \cite{SLD1, SLD2}, or in both \cite{CS1, BaFo, FSP1, BrSo}.

\begin{figure}
\includegraphics[scale=0.9]{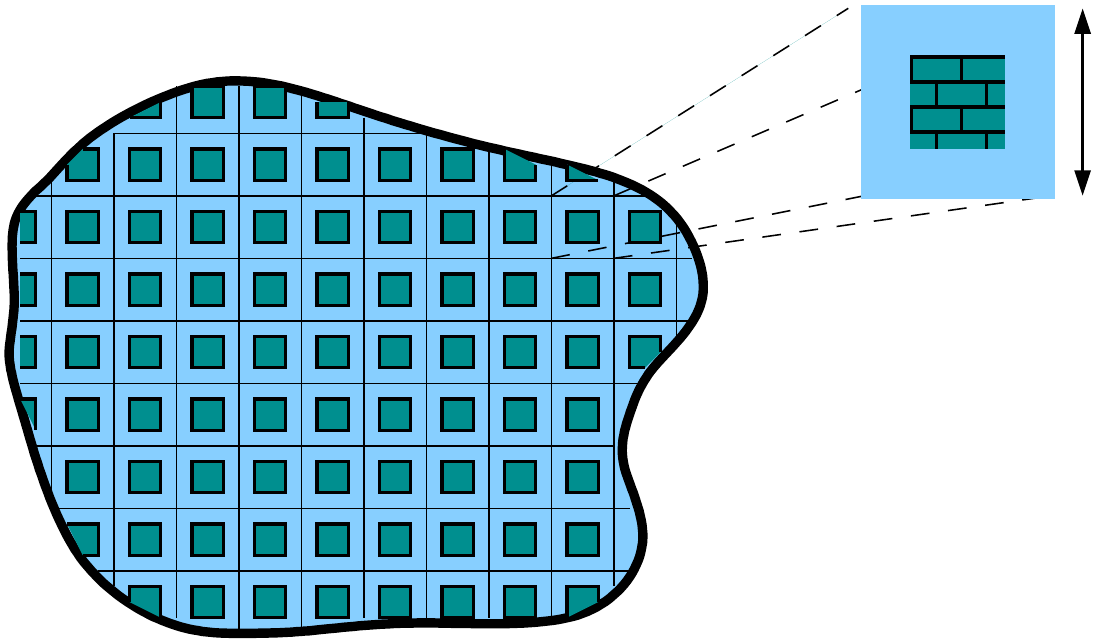}
\caption{Schematic of a periodic brittle material with weak inclusions}
\begin{picture}(0,0) 
\put(30,80){{$\Omega$}}
\put(145,175){{$\e$}}
\end{picture}
\end{figure}
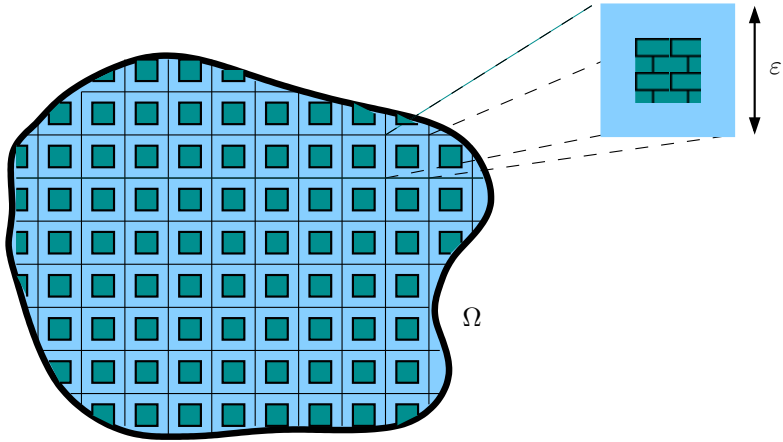

For the functional \eqref{Fe} the high-contrast behaviour is in the surface term. We show that depending 
on how small $\beta_\e$ is, with respect to the `critical' value $\e$, the effective behaviour of the high-contrast material is different. To see this we introduce the parameter 
$$
\ell:=\lim_{\e \to 0}\frac{\beta_\e}{\e}\in [0,+\infty],
$$
and characterise the $\Gamma$-limit of $F_\e$ in the case $\ell=0$ (subcritical regime), $\ell\in (0,+\infty)$ (critical regime), and $\ell=+\infty$ (supercritical regime).  
	
\subsection{Abstract $\Gamma$-convergence result and the choice of the convergence.} As a first step we prove that there exists an infinitesimal sequence $(\e_k)$ along which ${F}_{\varepsilon_k}$ $\Gamma$-converges to a functional ${F}^{\ell, (\e_k)}_{\rm{hom}}$ (depending on the sequence $(\e_k)$) which can be represented in an integral form as		
\begin{equation*}
			{F}^{\ell, (\e_k)}_{\rm{hom}}(u)= 
			\begin{cases} 
			\displaystyle
			\int_{\Omega}f^{\ell, (\e_k)}_{\rm{hom}}(\nabla u) dx 
			+ \int_{S_{u}} g^{\ell, (\e_k)}_{\rm{hom}}([u], \nu_u) d\mathcal{H}^{n-1}
			& \mbox{if} \ u \in GSBV^{2}(\Omega), \\ 
			+\infty
			& \mbox{otherwise in } L^{1}(\Omega), 
			\end{cases}
		\end{equation*}
where $f^{\ell, (\e_k)}_{\rm{hom}}$ and $g^{\ell, (\e_k)}_{\rm{hom}}$ depend on the mutual vanishing rate of $\beta_\varepsilon$ and $\e$, that is on $\ell$, and $GSBV^{2}(\Omega)$ is the space of generalised special functions with bounded variation.

In this step the choice of the convergence plays a crucial role and introduces some difficulties. The functional $F_\e$ in \eqref{Fe} is non-coercive with respect to the $L^1(\Omega)$ topology, due to the infinitesimal prefactor 
$\beta_\e$ for the measure of the jump set in the inclusions $\Omega \setminus \e P$.
In \cite{CS1}, where the authors considered the Mumford-Shah functional on a perforated domain, namely
		\begin{equation}\label{widehatFe-intro}
			\widehat{F}_{\varepsilon}(u):= 
				\begin{cases} 
				\displaystyle
				\int_{\Omega\cap\varepsilon P}|\nabla u|^{2} dx + \mathcal{H}^{n-1}(S_{u}\cap\varepsilon P)
				& \mbox{if} \ u_{|\Omega \cap \varepsilon P} \in GSBV^{2}(\Omega \cap \varepsilon P), \\ 
				+\infty
				& \mbox{otherwise in } L^{1}(\Omega),
			\end{cases}
		\end{equation}
the (even more prominent) lack of coerciveness was solved by means of an extension result. Namely, each sequence $(u_\e)$ with equibounded energy $\widehat{F}_{\varepsilon}$ was replaced by a new, improved sequence $\tilde u_\e$ which coincides with $u_\e$ outside the perforations and which is precompact in $L^1(\Omega)$. Since $\widehat{F}_{\varepsilon}(u_\e) = \widehat{F}_{\varepsilon}(\tilde u_\e)$, this substitution does not affect the energy, and hence it is natural to study the $\Gamma$-convergence of $\widehat{F}_\e$ with respect to the strong convergence in $L^1(\Omega)$ (see also \cite{BaFo, FSP1, BrSo}). 

The situation in the case of the functional $F_\e$ in \eqref{Fe} is quite different, since ${F}_{\varepsilon}(u_\e) \neq {F}_{\varepsilon}(\tilde u_\e)$. It is therefore necessary to formulate a compactness result for the original sequence $(u_\e)$. The convergence we introduce only looks at the behaviour of $(u_\e)$ on the matrix $\Omega \cap \e P$: We say that $(u_\e)$ converges to $u$ in $\Omega \cap \e P$ if there exists a sequence $(\tilde u_\e) \subset L^1(\Omega)$ such that $\tilde u_\e = u_\e$ in $\Omega\cap \varepsilon P$, and $\tilde u_\e$ converges to $u$ strongly in $L^1(\Omega)$ (see Definition \ref{def:conv}).  

This new convergence is natural since it guarantees the convergence of minimisers and minimum values of $F_\e$, up to the addition of a forcing term (see Proposition \ref{propo:conv_min}). It introduces, however, several difficulties. First of all, we can only work with a \textit{sequential} notion of $\Gamma$-convergence for $F_\e$. As a result, the so-called localisation method of $\Gamma$-convergence that is usually employed to prove the existence of a $\Gamma$-limit of integral type does not apply directly. In particular, the proof of the fundamental estimate, which is a crucial step to guarantee that the $\Gamma$-limit is a measure, requires some care (see Lemma \ref{lemma:FE}). Moreover, dealing with a weaker convergence (than the usual $L^1(\Omega)$) implies that there are more converging sequences, and therefore proving an optimal lower bound for $f^{\ell,(\e_k)}_{\rm}$ and $g^{\ell,(\e_k)}_{\rm}$ is more subtle. 

\subsection{The surface energy} We show in Theorem \ref{t:g-hom-g-hat} that the homogenized  surface integrand $g^{\ell, (\e_k)}_{\rm{hom}}$ does not depend on $\ell$ or on the subsequence $(\e_k)$ and 
\begin{equation}\label{ghat-intro}
g_{\rm{hom}}^\ell(\nu) = \hat g(\nu)
\end{equation}
for every $\ell\in [0,+\infty]$ and every $\nu\in \Sph^{n-1}$, where $\hat g$ denotes the limit surface density of the Mumford-Shah functional in a periodically perforated domain, $\widehat{F}_{\varepsilon}$, in \eqref{widehatFe-intro}. 
This result is surprising since the functionals $F_\e$ in \eqref{Fe} exhibit a high-contrast, $\beta_\e$-dependent behaviour in the surface term, while the limit surface energy density is independent of $\ell$ and 
coincides with the case of $\beta_\e=0$. In other words, the effective cost of introducing a crack in the material is the lowest possible, and coincides with the case where the inclusions 
are replaced by perforations. 

Note that, since $F_\e \geqslant \widehat F_\e$, the bound $g_{\rm{hom}}^\ell \geqslant \hat g$ is immediate. However, proving the opposite inequality is nontrivial and it requires extending
partitions of a perforated domain inside the perforations, without essentially increasing the perimeter of the partition.

\subsection{The volume energy}	
The homogenized  volume integrand $f^{\ell, (\e_k)}_{\rm{hom}}$, instead, shows a nontrivial dependence on $\ell$. The dependence on the subsequence $(\e_k)$, however, is only present in the critical regime $\ell\in (0,+\infty)$ and not in the extremal cases $\ell=0$ and $\ell=+\infty$.

Specifically, in the subcritical regime $\ell=0$,  for every $\xi\in\mathbb{R}^{n}$
		\begin{equation}\label{f0-intro}
			f^0_{\rm{hom}}(\xi)= \hat f(\xi),
	\end{equation}
where $\hat f$ denotes the limit volume density of $\widehat F_\e$ (Theorem \ref{t:f-hom-f-hat}). So, in the subcritical regime $\beta_\e\ll\e$, $F_\e$ and $\widehat{F}_{\varepsilon}$ are asymptotically equivalent, since `cutting out' all the inclusions has an infinitesimal energy-cost of order $\beta_\e/\e$ (given by the perimeter of the inclusion, which is proportional to $\beta_\e \e^{n-1}$, for each of the $1/\e^n$ $\e$-cells in $\Om$). In other words, having very weak inclusions is equivalent to having a perforated material.

In the other extremal regime\ie $\ell=+\infty$, the homogenized  volume integrand is the highest possible, namely $f^\infty_{\rm{hom}}(\xi) = |\xi|^2$ (Theorem \ref{id:f_infty}). In this case it is the upper bound 
$f^\infty_{\rm{hom}}(\xi) \leqslant |\xi|^2$ that is immediate; the difficulty in proving the opposite inequality is due to having to prove that every sequence $u_\e$ converging to $u_\xi=\xi \cdot x$ in $\Omega \cap \e P$ satisfies the lower bound
\begin{equation}\label{i:xi2}
\liminf_{\e\to 0} F_\e(u_\e) \geq \mathcal{L}^n(\Om) |\xi|^2.
\end{equation}
Note that \eqref{i:xi2} means that the microscopic cracks at scale $\varepsilon$ do not lower the elastic moduli of  the material. We prove the lower bound \eqref{i:xi2} by classifying each $\e$-cell in $\Omega$. Either the measure of the jump set of $u_\e$ in the cell is large, in which case we `cut out' the inclusion from the cell as in the subcritical case. Or, alternatively, the measure of the jump set of $u_\e$ in the cell is small, in which case the function is essentially smooth, thanks to an `Elimination Property' for the jump set due to De Giorgi, Carriero and Leaci \cite{DeGiorgi} (see also \cite{DMMS92}). Due to the relatively high cost of creating a fracture in the supercritical case, we show that it is energetically convenient to have no fracture at all in the majority of the cubes, and this gives \eqref{i:xi2}. 

In the critical case $\ell\in (0,+\infty)$, we cannot exclude that the volume energy density $f_{\rm{hom}}^{\ell,(\varepsilon_k)}$ depends on the subsequence $(\e_k)$ along which we prove $\Gamma$-convergence. Each (subsequence dependent) $f_{\rm{hom}}^{\ell,(\varepsilon_k)}$, however, satisfies some (subsequence independent) properties. First of all, 
$f_{\rm{hom}}^{\ell,(\varepsilon_k)} = |\xi|^2$ for small $|\xi|$; namely, for small $\xi$ the effective elastic behaviour of the material is the same as in the supercritical case. 
Moreover, unlike the extremal cases $\ell\in\{0,+\infty\}$, the volume energy density \textit{is not $2$-homogeneous} even if the volume density of $F_\varepsilon$ is. 
This shows emergence of non-standard constitutive laws in the homogenized  limit of high-contrast brittle materials.

\subsection{Comparison with previous work} This work has interesting similarities and differences with previous results on the homogenization of free-discontinuity energies. 

\subsubsection{Interaction of the volume and surface term in the $\Gamma$-limit} Our result shows that the volume and surface terms of $F_\e$ interact in the limit. This case is different from the analysis in \cite{CDMSZ}, where the authors devise a list of assumptions ensuring that the volume and surface energies do not interact in the homogenization of free-discontinuity functionals.  
The functionals $F_\e$ in \eqref{Fe} are not covered by the analysis in \cite{CDMSZ} due to their degenerate growth conditions. Such a degeneracy, and the consequent lack of coerciveness, 
however, is not sufficient to cause the interaction of the terms of the energy, since this was not the case for the functionals $\widehat{F}_\e$.

\subsubsection{Other high-contrast Mumford-Shah energies} The general form of high-contrast Mumford-Shah energies is 
\begin{equation*}
			{H}^{\alpha_\e,\beta_\e}_{\varepsilon}(u) 
			= \int_{\Omega\cap\varepsilon P}|\nabla u|^{2} dx + \alpha_\e\int_{\Omega\cap(\Omega\setminus\varepsilon P)}|\nabla u|^{2} dx 
			+ \mathcal{H}^{n-1}(S_{u}\cap\varepsilon P)
			+ \beta_{\varepsilon}\mathcal{H}^{n-1}(S_{u}\cap(\Omega\setminus\varepsilon P)),
\end{equation*}
where $\alpha_\e, \beta_\e\geq 0$, with either $\alpha_\e$ or $\beta_\e$ being infinitesimal for $\e\to 0$. Note that $\widehat{F}_\e = {H}^{0,0}_{\e}$, $F_\e = {H}^{1,\beta_\e}$, while the case studied in \cite{BLZ} corresponds 
to ${H}^{\alpha_\e,1}_{\varepsilon}$. In \cite{BLZ} the authors proved that the volume energy density of the $\Gamma$-limit of ${H}^{\alpha_\e,1}_{\varepsilon}$ is $\hat f$, regardless of the smallness of $\alpha_\e$ relative to $\e$. Combining this result with the identification of the $\Gamma$-limit of $F_\e = {H}^{1,\beta_\e}$ proved in the present paper, and of the $\Gamma$-limit of $\widehat F_\e = {H}^{0,0}$ in \cite{CS1, BaFo, FSP1}, we can deduce the expression of the $\Gamma$-limit of $H^{\alpha_\e,\beta_\e}_\e$, when both $\alpha_\e\to 0$ and $\beta_\e\to 0$. This follows by letting $\e\to 0$ in the estimate 
\begin{equation}\label{est:Hab}
{H}^{0,0}_{\e}(u) \leqslant H^{\alpha_\e,\beta_\e}_\e(u) \leqslant \min\{{H}_\e^{1,\beta_\e}(u), {H}^{\alpha_\e,1}_{\varepsilon}(u)\}.
\end{equation}
Indeed, the limit volume energy density in the left- and right-hand sides of the inequality \eqref{est:Hab} is $\hat f$, and the limit surface energy density in both sides of the inequality is $\hat g$, so that the $\Gamma$-limit of $H^{\alpha_\e,\beta_\e}_\e$ is the same as that of ${H}^{0,0}_{\e}$\ie the same as the $\Gamma$-limit of the perforated Mumford-Shah functional $\widehat F_\e$ (see also \cite{BrSo}).

\subsection{Structure of the paper} In Section \ref{Sec:prel} we introduce the notation used in the paper and we recall some previous results on the homogenization of Mumford-Shah-type energies. 
Section \ref{Sect:setting} is devoted to the statement of our main $\Gamma$-convergence result, Theorem \ref{T1}, and of its consequences. The proof of Theorem \ref{T1} is split into the remaining three
sections: The abstract $\Gamma$-convergence and integral representation result is proved in Section \ref{sect:abstract}, while Sections \ref{sect:g} and \ref{sect:f} are devoted to the characterisation 
of the surface and volume energy densities, respectively.

\section{Notation and preliminaries}\label{Sec:prel}

In this section we fix the notation and recall some definitions and results that we are going to use throughout the paper. 

\subsection{Notation} Let $n\geqslant 2$ and let $\Omega\subset \R^n$ be open bounded and with Lipschitz boundary. 
We denote with $\mathcal{A}(\Omega)$ the class of all open subsets of $\Omega$, and with $\mathcal{B}(\Omega)$ the $\sigma$-algebra of Borel sets in $\Omega$.
The $n$-dimensional Lebesgue measure is denoted by $\mathcal{L}^{n}$, and the $(n-1)$-dimensional Hausdorff measure by $\mathcal{H}^{n-1}$. For $U$, $V\in\mathcal{B}(\Omega)$ we set 
$\mathcal{H}^{n-1}\mres U(V):= \mathcal{H}^{n-1}(U \cap V)$.
For every $x \in \mathbb{R}^{n}$ and $r>0$, $B_r(x)$ will be the open ball with centre $x$ and radius $r$, with $B_r:=B_r(0)$; for $0<s<r$ we also set  
$B_{s,r}:= B_r \setminus \overline B_s$. The boundary of the ball $B_1$ will be denoted with  $\Sph^{n-1}$.

For $r>0$, $Q_r$ denotes the open cube centred at the origin, with side-length $r$. We write $Q=Q_1$. For $0<s<r$ we also set  
$Q_{s,r}:= Q_r \setminus \overline Q_s$. 
	
We also define the periodic set
\begin{equation}\label{def:P}
P:=\mathbb{R}^{n}\setminus \left(\bigcup_{i \in \mathbb{Z}^{n}} (i+Q_{1/2})\right).
\end{equation}

\medskip

The functional setting for our analysis is that of \textit{special functions of bounded variation} in 
$\Omega$\ie	
\begin{equation*}
SBV(\Omega):= \{u \in BV(\Omega): \ Du = \nabla u \mathcal{L}^{n} + (u^{+} - u^{-})\nu_{u}d\mathcal{H}^{n-1} \mres S_{u}\}.
\end{equation*}
Here $S_{u}$ denotes the discontinuity set of $u$, $\nu_{u}$ is the generalised normal to $S_{u}$, $u^{+}$ and $u^{-}$ are the traces of $u$ on both sides of $S_{u}$. 
More precisely, we work with the following vector subspace of $SBV(\Omega)$:
\begin{equation*}
SBV^{2}(\Omega):= \{u \in SBV(\Omega): \ \nabla u\in L^{2}(\Omega) \text{ and } \mathcal{H}^{n-1}(S_{u})<+\infty\}.
\end{equation*}
We consider also the larger space of \textit{generalised special functions of bounded variations} in $\Omega$, 
\begin{equation*}
GSBV(\Omega):= \{u \in L^{1}(\Omega): \ (u \wedge m) \vee (-m) \in SBV(\Omega) \text{ for all } m\in \N\}.
\end{equation*}
By analogy with the case of $SBV$ functions, we write
\begin{equation*}
GSBV^{2}(\Omega):= \{u \in GSBV(\Omega): \ \nabla u\in L^{2}(\Omega) \text{ and } \mathcal{H}^{n-1}(S_{u})<+\infty\}.
\end{equation*}
We also consider the space
$$
SBV^{\mathrm{pc}}(\Om):=\{u\in SBV(\Om): \nabla u=0 \,\, \mathcal{L}^{n}\textrm{-a.e.}, \,\mathcal{H}^{n-1}(S_u)<+\infty\};
$$
it is known (see \cite[Theorem 4.23]{AFP}) that every $u$ in $SBV^{\mathrm{pc}}(\Om)\cap L^\infty(\Om)$ is piecewise constant in the sense of \cite[Definition 4.21]{AFP}, namely 
there exists a Caccioppoli partition $(E_i)$ of $\Om$ such that $u$ is constant $\mathcal{L}^n$-a.e.\ in each set $E_i$. 

Moreover, we set
$$
\mathcal{P}(\Omega):=\{u \in SBV^{\textrm{pc}}(\Omega)\colon u(x)\in \{0,1\} \ \mathcal{L}^n\textrm{- a.e. in } \Om\}.
$$

\subsection{Mumford-Shah-type energies} For $\varepsilon>0$ let $0 \leqslant \alpha_{\varepsilon}, \beta_{\varepsilon} \leqslant 1$; we define the functional 
$H_{\varepsilon}: L^1(\Omega)\to [0,+\infty]$ as follows
\begin{equation}\label{MSab}
H_{\varepsilon}(u) 
:= \begin{cases}
\displaystyle\int_{\Omega}a_\e\Big(\frac{x}{\e}\Big)|\nabla u|^{2} dx + \int_{S_u\cap \Omega} b_\e\Big(\frac{x}{\e}\Big) d\mathcal{H}^{n-1}(x)
& \quad  \text{if}\; u\in GSBV^2(\Omega) \\
+\infty & \quad \text{otherwise},
\end{cases}
\end{equation}
where $a_\e, b_\e: \R^n \to \R$ are the $Q$-periodic functions defined as
$$
a_\e(y):= \begin{cases} 1 \quad &\textrm{if } y \in Q_{\frac{1}{2},1}\\ 
\alpha_\e \quad &\textrm{if } y \in Q_{\frac{1}{2}}
\end{cases},
\qquad 
b_\e(y):= \begin{cases} 1 \quad &\textrm{if } y \in Q_{\frac{1}{2},1}\\ 
\beta_\e \quad &\textrm{if } y \in Q_{\frac{1}{2}}
\end{cases},
$$
on the periodicity cell $Q$. The particular choice of the geometry of the inclusions (or of the set $P$) is not relevant for the subsequent analysis: Instead of $Q_{1/2}$ we could consider any Lipschitz open subset $U$ of $Q$, with $U\subset\subset Q$.
In what follows however we only consider $P$ as in \eqref{def:P} for the sake of the exposition.

\subsubsection{The extreme cases $\alpha_\e=\beta_\e=1$ and $\alpha_\e=\beta_\e=0$.} The case $\alpha_\e=\beta_\e=1$ corresponds to the ($\e$-independent) Mumford-Shah functional, which we denote with
		\begin{equation}
			MS(u)
			:=\begin{cases}
			\displaystyle
			\int_{\Omega}|\nabla u|^{2} dx + \mathcal{H}^{n-1}(S_{u}\cap\Omega) 
			& \text{if } u\in GSBV^{2}(\Omega),\\
			+\infty 
			& \text{otherwise in } L^{1}(\Omega).
			\end{cases}
			\label{MS1}
		\end{equation}	
The other extreme case $\alpha_\varepsilon = \beta_\varepsilon =0$ corresponds to the Mumford-Shah functional on a periodically perforated domain, which we denote with
		\begin{equation}\label{widehatFe}
			\widehat{F}_{\varepsilon}(u) := 
				\begin{cases} 
				\displaystyle
				\int_{\Omega\cap\varepsilon P}|\nabla u|^{2} dx + \mathcal{H}^{n-1}(S_{u}\cap\varepsilon P)
				& \mbox{if} \ u_{|\Omega \cap \varepsilon P} \in GSBV^{2}(\Omega \cap \varepsilon P), \\ 
				+\infty
				& \mbox{otherwise in } L^{1}(\Omega). 
			\end{cases}
		\end{equation}
We recall that the $\Gamma$-limit of $\widehat{F}_{\varepsilon}$ with respect to the $L^1(\Omega)$-topology, which has been studied in \cite{FSP1, CS1, BaFo, BrSo}, is the anisotropic free-discontinuity functional $\widehat{F} : L^{1}(\Omega) \rightarrow [0,+\infty]$  defined as
		\begin{equation}\label{def:Fhat}
			\widehat{F}(u) := 
				\begin{cases} 
				\displaystyle
				\int_{\Omega}\hat{f}(\nabla u)\dx + \int_{S_{u}\cap\Omega} \hat{g}(\nu_u)\, d\mathcal{H}^{n-1}
				& \mbox{if} \ u \in GSBV^{2}(\Omega), \\ 
				+\infty
				& \mbox{otherwise in } L^{1}(\Omega),
			\end{cases}
		\end{equation}
(see, e.g., \cite[Theorem 4]{BaFo}). In \eqref{def:Fhat}, the volume density $\hat f \colon {\R}^{n} \to [0,+\infty)$ is the quadratic form given by  
		\begin{equation}
			\hat{f}(\xi) := \inf\left\{\int_{Q \cap P}|\xi + D w|^{2}  dx: \ w \in H_{\textrm{per}}^{1}(Q \cap P)\right\},
			\label{F}
		\end{equation}
moreover, there exists a constant $c_{1}:=c_1(n,P)>0$ such that
\begin{equation} 
c_{1}|\xi|^{2}\leqslant\hat{f}(\xi)\leqslant|\xi|^{2} \ \text{ for every } \xi\in\mathbb{R}^{n}.
\label{B1}
\end{equation} 
The surface density $\hat g\colon  \Sph^{n-1} \to [0,+\infty)$ is defined as
		\begin{multline}\label{G}
			\hat{g}(\nu) := \lim_{\varepsilon \rightarrow 0^{+}} \inf
			\Big\{ \mathcal{H}^{n-1}(S_{w} \cap (Q^\nu \cap \varepsilon P)) : \ 
			w\in \mathcal P(Q^{\nu} \cap \e P), \\ 
			w = u_{0,1}^\nu \text{ in a neighbourhood of } \partial Q^{\nu}\Big\},
		\end{multline}
	where $Q^{\nu}$ stands for the unit cube centred in $0$ with one face orthogonal to $\nu$ and
\begin{equation*}
			u_{0,1}^{\nu}(x) = \begin{cases} 
			1 & \mbox{if} \ \langle x,\nu\rangle \geqslant 0 \\ 
			0 & \mbox{if} \ \langle x,\nu\rangle < 0.
			\end{cases}
		\end{equation*}
It can be seen that the function $\hat g$ in \eqref{G} is continuous on $\Sph^{n-1}$ and satisfies 
\begin{equation}
c_{2}\leqslant\hat{g}(\nu)\leqslant 1 \ \text{ for every } \nu\in \Sph^{n-1},
\label{B2}
\end{equation}
for some constant $c_{2}:=c_2(n,P)>0$.
Therefore, gathering \eqref{B1} and \eqref{B2} we obtain the following lower bound for $\widehat F$  
		\begin{equation}
			\min\{c_{1},c_{2}\}MS \leqslant \widehat{F}.
			\label{BB2}
		\end{equation}

\begin{rem}	
By changing variables, we can equivalently write
		\begin{multline*}
			\hat{g}(\nu) = \lim_{t \rightarrow + \infty} \dfrac{1}{t^{n-1}}
			\min \{ 
			 \mathcal{H}^{n-1}(S_{w} \cap P \cap tQ^{\nu}) \ : \
			u \in \mathcal P(tQ^{\nu} \cap P), \\ 
			 w = u_{0,1}^\nu \text{ in a neighbourhood of } \partial (tQ^{\nu})
			\}.
		\end{multline*}
\end{rem}

\begin{rem}\label{rem:convBS}
In \cite{BrSo} Braides and Solci proved, among other things, that $\widehat F_\e$ $\Gamma$-converges to $\widehat F$ also with respect to the following convergence, that is weaker than convergence in $L^1(\Om)$: Given $u_\e, u \in L^1(\Om)$ we say that $u_\e$ converges to $u$ if 
$$
u_\e \chi_{\Om\cap \e P} \rightharpoonup C_P u \quad \textrm{weakly in } L^1(\Om),
$$
where $C_P := \mathcal{L}^n(Q\cap P)= \mathcal{L}^n(Q_{\frac{1}{2},1})=1-\big(\frac{1}{2}\big)^n$.
\end{rem}

\subsubsection{High-contrast Mumford-Shah functionals} 	

This is the case of the functional $H_\e$ in \eqref{MSab} where either $\alpha_\e \to 0$ or $\beta_\e\to 0$, as $\e \to 0$. For this choice of $\alpha_\e$ and $\beta_\e$, $H_\e$ represents the energy associated to a high-contrast composite material, the latter being characterised by two constituents (the matrix and the inclusions) with significantly different mechanical properties. 

\smallskip

The case where both $\alpha_{\varepsilon} \to 0$ and $\beta_{\varepsilon}\to 0$, as $\e\to 0$ was considered by Braides and Solci in \cite{BrSo}. In this case, independently of the vanishing rate of $\alpha_\varepsilon$ and $\beta_\varepsilon$, the functionals $H_{\varepsilon}$
$\Gamma$-converge to $\widehat F$ with respect to the convergence introduced in Remark \ref{rem:convBS}. This means that, no matter how small the weights $\alpha_\e$ and $\beta_\e$ are, as long as they are \textit{both} infinitesimal, the effective behaviour of the functional $H_\e$ is the same as for $\widehat{F}_\e$ in \eqref{widehatFe}, namely it is the same as for $\alpha_\e=\beta_\e=0$.

\smallskip

The case $\alpha_{\varepsilon} \to 0$ and $\beta_{\varepsilon}=1$ was treated in \cite{BLZ}, where the authors focus on the critical case $\alpha_\varepsilon=\varepsilon$ and study the $\Gamma$-convergence of $H_{\varepsilon}$ with respect to the strong convergence in $L^{1}(\Omega)$. They prove that the $\Gamma$-limit of $H_\e$, which exists up to a subsequence, is of free-discontinuity type. More precisely, the limit volume density is the function $\hat{f}$ in \eqref{F} corresponding to the case $\alpha_\e=0$. The surface energy density, instead, depends non trivially on the jump opening $[u]$, even if the surface energy density in $H_\e$ is identically equal to $1$. In particular, the case considered in \cite{BLZ} is an example where volume and surface energies interact in the $\Gamma$-limit thus giving rise to a homogenized  effective energy of completely different nature with respect to the microscopic ones. 
			
\smallskip
			
Here we consider the complementary case to \cite{BLZ}, namely in $H_\e$ we choose $\alpha_{\varepsilon} = 1$ and let $\beta_{\varepsilon} \to 0$ as $\e\to 0$. This choice corresponds to having inclusions which are much more brittle than the matrix, but whose elastic behaviour is the same as that of the matrix.







\section{Setting of the problem and statement of the main result} \label{Sect:setting}
In this section we state our main result, that is a $\Gamma$-convergence theorem for the functionals $F_\e \colon L^1(\Om) \longrightarrow [0,+\infty]$ defined as 
\begin{equation}
			{F}_{\varepsilon}(u) := 
			\begin{cases} 
			\displaystyle
			\int_{\Omega}|\nabla u|^{2} \dx
			+ \mathcal{H}^{n-1}(S_u\cap(\Omega\cap\varepsilon P))
			+ \beta_{\varepsilon}\mathcal{H}^{n-1}(S_u\cap(\Omega \setminus \varepsilon P)) 
			& \mbox{if} \ u \in GSBV^{2}(\Omega), \\ 
			+\infty 
			& \mbox{otherwise in } L^{1}(\Omega),
			\end{cases}
			\label{GG1}
		\end{equation}
where $\beta_\e \searrow 0$ as $\e$ tends to zero. We analyse the asymptotic behaviour of $F_\e$ in three possible scaling-regimes\ie $\beta_\e \ll \e$ (subcritical regime), $\beta_\e \sim \e$ (critical regime), and $\beta_\e \gg \e$ (supercritical regime). We note that trivially	
		\begin{equation}
			\widehat{F}_{\varepsilon} \leqslant F_{\varepsilon} \leqslant MS,
			\label{BB1}
		\end{equation}
where $\widehat{F}_{\varepsilon}$ and $MS$ are defined in \eqref{widehatFe} and \eqref{MS1}, respectively.

\medskip

For the $\Gamma$-convergence result to be meaningful we first need to identify a notion of convergence on the space $L^1(\Om)$ for which the equicoerciveness of the functionals $F_\e$ is guaranteed. 

\subsection{Choice of the convergence} 
Let $(u_{\varepsilon})\subset L^1(\Omega)$ be a sequence satisfying 
\begin{equation}\label{bounds-conv}
\sup_\e \|u_\e\|_{L^\infty(\Om\cap \e P)}<+\infty \quad \text{and} \quad \sup_{\e}F_\e(u_\e)<+\infty.
\end{equation}
Since $\beta_\e$ is a vanishing sequence, the uniform bound on $F_\e(u_\e)$ gives no control on the surface term $\mathcal{H}^{n-1}(S_{u_\e}\cap(\Omega \setminus \varepsilon P))$ so that, in particular, the sequence $(u_\e)$ is not uniformly bounded in $BV(\Om)$. 
However, by the extension result \cite[Theorem 1.1]{CS1} we can find a new sequence $\tilde u_{\e} \subset L^1(\Om)$ satisfying
\begin{equation}\label{extension-Lu}
\begin{cases}
\vspace{.2cm}
\tilde u_{\e} = u_\e \quad \text{in}\; \Omega \cap \varepsilon P\\
\vspace{.2cm}
\sup_\e \|\tilde u_\e\|_{L^\infty(\Om)}<+\infty\\
\ds MS(\tilde u_\e) \leqslant C F_\e(u_\e),
\end{cases}
\end{equation}
for some $C>0$. Therefore, by combining \eqref{bounds-conv}-\eqref{extension-Lu} with the Ambrosio Compactness Theorem (see \textit{e.g.}, \cite[Theorem 4.8]{AFP}) we deduce that there exist $u \in SBV^2(\Omega)$ and a subsequence of $(\tilde u_\e)$, (not relabelled) such that $\tilde u_\e \to u$ in $L^1(\Om)$. 
As a consequence we get
\begin{align*}
\int_{\Om \cap \e P}|u_\e-u|\dx = \int_{\Om \cap \e P}|\tilde u_\e-u|\dx \leqslant \int_{\Om}|\tilde u_\e-u| \dx \to 0 \quad \textrm{as } \e\to 0.
\end{align*}
The above observation motivates the following definition. 

\begin{defi}[Convergence]\label{def:conv} 
Let $(u_\e)$ be a sequence in $L^1(\Om)$. We say that $(u_\e)$ converges in $\Om\cap \e P$ to a function $u\in L^1(\Om)$, and we write $u_\e \to u$,  
if there exists a sequence $(\tilde u_\e) \subset L^1(\Omega)$ such that $\tilde u_\e = u_\e$ in $\Omega \cap \e P$, and $\tilde u_\e$ converges to $u$ strongly in $L^1(\Omega)$.
\end{defi}

\begin{rem}
{\rm We observe that Definition \ref{def:conv} is well-posed. Indeed let $(\tilde u_{1,\e}), (\tilde u_{2,\e})\subset L^1(\Omega)$
be such that $\tilde u_{1,\e} = \tilde u_{2,\e} = u_\e$ in $\Om \cap \e P$ and $\tilde u_{1,\e} \to u_1$ and $\tilde u_{2,\e} \to u_2$ strongly in $L^1(\Omega)$. Then
$$
0= \lim_{\e \to 0}\int_{\Om \cap \e P}|\tilde u_{1,\e} - \tilde u_{2,\e}| \dx = \lim_{\e \to 0}\int_{\Om}|\tilde u_{1,\e} - \tilde u_{2,\e}|\chi_{\Om \cap \e P} \dx = C_P \int_{\Om}|u_1 - u_2| \dx, 
$$
hence, since $C_P = \mathcal{L}^n(Q_{\frac{1}{2},1})>0$ we necessarily have $u_1=u_2$.
}
\end{rem}

\begin{rem}\label{conv-deb-troncata}
{\rm Let $(u_\e) \subset L^1(\Om)$ be such that $u_\e \to u$ in the sense of Definition \ref{def:conv}, for some $u\in L^1(\Om)$. It is immediate to see that
\begin{itemize}
\item[i.] $\lim_\e \|u_\e -u\|_{L^1(\Om\cap \e P)}=0$; 
\item [ii.] the sequence $(u_\e \chi_{\Om \cap \e P})\subset L^1(\Om)$ converges to $C_P u$ weakly in $L^1(\Om)$;
\item[iii.] if $(u_\e^m)$ denotes the sequence of truncated functions of $u_\e$ at level $m\in \N$, then $u_\e^m \to u^m$ in the sense of Definition \ref{def:conv}, where $u^m$ denotes the truncated function of $u$ at level $m$.
\end{itemize}
}
\end{rem}

In view of the above considerations, in what follows we study the $\Gamma$-convergence of the functionals $F_\e$ with respect to the convergence as in Definition \ref{def:conv}. To this end we give the following definition.  

\begin{defi}[Sequential $\Gamma$-convergence]\label{def:G-conv}
Let $F_\e, F \colon L^1(\Om) \longrightarrow [0,+\infty]$; we say that the functionals $F_\e$ $\Gamma$-converge to $F$ with respect to the convergence as in Definition \ref{def:conv} if for every $u\in L^1(\Om)$ the two following conditions are satisfied:

\smallskip

$(i)$ (Ansatz-free lower bound) For every $(u_\e)\subset L^1(\Om)$ with $u_\e \to u$ in the sense of Definition \ref{def:conv} we have
$$
F(u) \leqslant \liminf_{\e \to 0} F_\e(u_\e);
$$

$(ii)$ (Existence of a recovery sequence) There exists $(\bar u_\e) \subset L^1(\Om)$ with $\bar u_\e \to u$ in the sense of Definition \ref{def:conv} such that
$$
F(u) \geqslant \limsup_{\e \to 0} F_\e(u_\e).
$$
\end{defi}

\begin{rem}\label{rem:Fps_min}
{\rm It is easy to check that $F$ is lower semicontinuous with respect to the convergence as in Definition \ref{def:conv}
and thus, in particular, with respect to the strong $L^1(\Om)$ convergence.}
\end{rem}

For every $u\in L^1(\Om)$ we set
\begin{equation}\label{def:G-li}
\Gamma\hbox{-}\liminf_{\e\to 0} F_{\e}(u):=\inf\Big\{\liminf_{\e\to 0}F_\e(u_\e) \colon u_\e \to u \Big\}
\end{equation}
and 
\begin{equation}\label{def:G-ls}
\Gamma\hbox{-}\limsup_{\e\to 0} F_{\e}(u):=\inf\Big\{\limsup_{\e\to 0}F_\e(u_\e) \colon u_\e \to u \Big\},
\end{equation}
where the convergence $u_\e \to u$ is understood in the sense of Definition \ref{def:conv}. We also introduce the more compact notation 
$F'(u):=\Gamma\hbox{-}\liminf_{\e} F_{\e}(u)$ and $F''(u):=\Gamma\hbox{-}\limsup_{\e} F_{\e}(u)$. Then, it is immediate to see that 
Definition \ref{def:G-conv} is equivalent to $F'=F''=F$ in $L^1(\Om)$. Further, it can be easily shown that the infima in \eqref{def:G-li} and \eqref{def:G-ls} are actually minima. 

\subsection{Compactness and domain of the $\Gamma$-limit} 

If not otherwise specified, in all that follows the $\Gamma$-convergence will be always understood in the sense of Definition \ref{def:G-conv}.

\medskip

The following proposition shows that the domain of the $\Gamma$-limit of $F_\e$ (if it exists) is the space $GSBV^2(\Om)$.

\begin{propo}[Domain of the $\Gamma$-limit]
Let $(u_{\varepsilon})\subset L^1(\Omega)$ be such that $\sup_{\e}F_\e(u_\e)<+\infty$. Assume moreover that $u_\e$ converges to $u$ in the sense of Definition \ref{def:conv}, then $u\in GSBV^2(\Om)$.
\end{propo}

\begin{proof}
By Definition \ref{def:conv} there exists $(\tilde u_\e) \subset L^1(\Omega)$ such that $\tilde u_\e = u_\e$ in $\Omega \cap \e P$, and $\tilde u_\e$ converges to $u$ strongly in $L^1(\Omega)$. Let $m\in \N$ and let $u_\e^m$ be the truncated function of $u_\e$ at level $m$. 
Appealing to \cite[Theorem 1.1]{CS1} we deduce the existence of a sequence $(\tilde u_\e^m)$ such that $\tilde u_\e^m=u^m_\e$ in $\Om\cap \e P$ and 
$$
MS(\tilde u_\e^m) \leqslant C F_\e (u_\e^m) \leqslant C F_\e(u_\e),
$$
for some $C>0$, where to establish the last inequality we have used the fact that $F_\e$ decreases by truncations. Therefore, in view of the bound on the energy, invoking \cite[Theorem 4.8]{AFP} yields the existence of a function $v\in SBV^2(\Om)$ and a subsequence of $(\tilde u_\e^m)$ (not relabelled) such that $\tilde u_\e^m \to v$ strongly in $L^1(\Om)$. For every fixed $m\in \N$ we have
$$
0= \lim_{\e \to 0}\int_{\Om \cap \e P}|\tilde u_\e^m - u_\e^m| \dx =  \lim_{\e \to 0}\int_{\Om}|\tilde u_\e^m - {(\tilde u_\e)^m}|\chi_{\Om \cap \e P} \dx = C_P \int_{\Om}|v - u^m| \dx, 
$$
therefore $v=u^m$ and $u^m \in SBV^2(\Om)$ for every $m\in \N$. This implies that $u\in GSBV^2(\Om)$.
\end{proof}

\subsection{$\Gamma$-convergence.}
Our main result is the following theorem.

\begin{theo}\label{T1}
	Let $F_{\varepsilon} \colon L^1(\Om) \longrightarrow (0,+\infty]$ be the functionals defined in \eqref{GG1} and let $
\ell:=\lim_{\e\to 0}\frac{\beta_\e}{\e}$. Then for every sequence of positive numbers decreasing to zero, there exists a subsequence $(\e_k)$ such that the functionals $F_{\e_k}$ $\Gamma$-converge to $F^\ell_{\rm hom} : L^{1}(\Om)\longrightarrow [0,+\infty)$ defined by
		\begin{equation*}
			F^\ell_{\rm hom}(u) := 
			\begin{cases} 
			\displaystyle
			\int_{\Omega}f^\ell_{\rm hom}(\nabla u) \dx 
			+ \int_{S_{u}\cap\Omega} g^\ell_{\rm hom}(\nu_u)\, d\mathcal{H}^{n-1}
			& \mbox{if} \ u \in GSBV^{2}(\Omega), \\ 
			+\infty
			& \mbox{otherwise in } L^{1}(\Omega). 
			\end{cases}
		\end{equation*}
Moreover, the function $g^\ell_{\hom}$ is independent of $\ell$ and it coincides with $\hat{g}$ as in \eqref{G}, while the function $f^\ell_{\hom}$ satisfies, for every $\xi\in\mathbb{R}^{n}$ and every $\ell\in[0,+\infty]$, the bounds
     	\begin{equation*}
			\hat{f}(\xi) 
			\leqslant f^\ell_{\rm{hom}}(\xi)
			\leqslant|\xi|^{2}, 			
		\end{equation*}
		where $\hat{f}$ is the quadratic form in \eqref{F}. Furthermore, in the extreme regimes $\ell=0$ and $\ell=+\infty$ we have, respectively, $f^0_{\rm{hom}}(\xi)=\hat f(\xi)$ and $f^\infty_{\rm{hom}}(\xi)=|\xi|^2$ for every $\xi \in \R^n$. Therefore for $\ell=0$ and $\ell=+\infty$ the whole sequence $(F_{\e})$ $\Gamma$-converges. 	
\end{theo}

We divide the proof of Theorem \ref{T1} into two main steps carried out, respectively, in Section \ref{sect:abstract} and in Sections \ref{sect:g}-\ref{sect:f}.
Specifically, the first step deals with the existence of a subsequence of $F_\e$ which $\Gamma$-converges to a homogeneous free-discontinuity functional of the form
$$
\int_\Om f^\ell(\nabla u)\dx + \int_{S_u}g^\ell(u^+-u^-, \nu_u)\,d\mathcal H^{n-1}.
$$
Then in the second step we identify the integrands $g^\ell$ and $f^\ell$ in the three different scaling regimes. 

\medskip

\subsection{Convergence of minimisation problems}
Thanks to the $\Gamma$-convergence result Theorem \ref{T1}, we are able to establish a convergence result for the minimisation problems associated with a suitable perturbation of the functionals $F_\e$.

\medskip

Let $g \in L^{\infty}(\Omega)$, and let $G_{\varepsilon} : L^{1}(\Omega) \longrightarrow [0,+\infty)$ be the functionals defined as
	\begin{equation*}
		G_{\varepsilon}(u):= \int_{\Omega \cap \varepsilon P}|u(x)-g(x)| \ dx.
	\end{equation*} 
Let moreover $G: L^{1}(\Omega) \longrightarrow [0,+\infty)$ be given by
\begin{equation*}
G(u):= C_P\int_{\Omega}|u(x)-g(x)| \ dx,
\end{equation*}
where $C_P:= \mathcal{L}^n(Q_{\frac{1}{2},1})$. 

Let $\e>0$ be fixed; we start by observing that $F_\e + G_\e$ is lower-semicontinuous with respect to the strong $L^1(\Om)$-convergence. Moreover, since $F_\e$ decreases by truncations, we can readily deduce that a minimising sequence $(u_j)$ for  
$F_\e + G_\e$ satisfies the uniform bound $\|u_j\|_{L^\infty(\Om)} \leqslant \|g\|_{L^\infty(\Om)}$. The latter allows us to
to invoke the compactness result \cite[Theorem 4.7]{AFP} and thus to use the direct methods to deduce the existence of a minimiser $u_\e \in SBV^2(\Om)$ for $F_\e + G_\e$. 

\medskip

The following proposition establishes a convergence result for minimisers and minimum values of $F_{\varepsilon} + G_{\varepsilon}$.

\begin{propo}\label{propo:conv_min}
Let $(\e_k)$ and $F^\ell_{\rm hom}$ be, respectively, the subsequence and the functional whose existence is established in Theorem \ref{T1}. Let $k\in \N$ be fixed and let $u_k \in L^1(\Om)$ be a solution to 
$$
m_{k}:= \min\big\{F_{\e_k}(u)+ G_{\e_k}(u): u\in L^1(\Om) \big\}.
$$ 
Then, up to subsequences (not relabelled), $u_k$ converges in the sense of Definition \ref{def:conv} to a function $\bar u \in SBV^2(\Om)$ which solves     
$$
m:= \min\big\{F^\ell_{{\rm hom}}(u)+ G(u): u\in L^1(\Om) \big\}.
$$ 
Moreover we have $m_{k} \to m$, as $k\to +\infty$. 
\end{propo}

\begin{proof}
Let $(u_k)\subset L^1(\Om)$ be a sequence of minimisers for $F_{\e_k}+ G_{\e_k}$\ie
$$
m_k=\min\big\{F_{\e_k}(u_k)+ G_{\e_k}(u_k): u\in L^1(\Om) \big\},
$$
hence in particular $\sup_{k\in \N}F_{\e_k}(u_k)<+\infty$. A truncation argument also yields $\|u_k\|_{L^\infty(\Om)} \leqslant \|g\|_{L^\infty(\Om)}$, so that we can readily deduce the existence of a subsequence of $(u_k)$ (not relabelled) and a function $\bar u \in SBV^2(\Om)$ such that $u_k \to \bar u$ in the sense of Definition \ref{def:conv}. Thus, in particular
$$
\lim_{k \to +\infty}\int_{\Om \cap \e_k P} |u_k -g|\dx = C_P\int_{\Om} |\bar u -g|\dx=G(\bar u),
$$
therefore in view of Theorem \ref{T1} we get
\begin{eqnarray}\label{mk-uno}
F^\ell_{\rm hom}(\bar u)+ G(\bar u)  \leqslant \liminf_{k \to +\infty} \Big(F_{\e_k}(u_k)+ G_{\e_k}(u_k)\Big)
 =\liminf_{k \to +\infty} m_k.
\end{eqnarray}
Now let $w\in SBV^2(\Om)$ be arbitrary; again appealing to Theorem \ref{T1} we can find $(w_k) \subset L^1(\Om)$ such that $w_k \to w$ in the sense of Definition \ref{def:conv} and 
$$
\lim_{k \to +\infty} F_{\e_k}(w_k) = F^\ell_{\rm hom}(w). 
$$
Since
$$
\lim_{k \to +\infty}\int_{\Om \cap \e P} |w_k -g|\dx = G(w),
$$
we immediately deduce
\begin{eqnarray}\label{mk-due}
\limsup_{k \to +\infty} m_k  \leqslant \limsup_{k \to +\infty} \Big(F_{\e_k}(w_k)+ G_{\e_k}(w_k)\Big)= F^\ell_{\rm hom}(w)+G(w). 
\end{eqnarray}
Thus, putting together \eqref{mk-uno} and \eqref{mk-due} we obtain
\begin{eqnarray*}
F^\ell_{\rm hom}(\bar u)+ G(\bar u) &\leqslant& \liminf_{k \to +\infty} m_k
\leqslant \limsup_{k \to +\infty} m_k
\\
 &\leqslant& F^\ell_{\rm hom}(w)+ G(w),
\end{eqnarray*}
hence by the arbitrariness of $w$ we deduce that $\bar u$ is a minimiser for $F^\ell_{\rm hom}+G$. Finally, taking $w=\bar u$ also implies $m_k \to m$. Since moreover this limit does not depend on the subsequence, the convergence holds true for the whole $(m_k)$. 
\end{proof}

\section{$\Gamma$-convergence and integral representation} \label{sect:abstract}
In this section we prove the existence of a $\Gamma$-convergent subsequence of $F_\e$ to a homogeneous free-discontinuity functional of the form
$$
\int_\Om f^\ell(\nabla u)\dx + \int_{S_u}g^\ell(u^+-u^-, \nu_u)\,d\mathcal H^{n-1}.
$$
To do so, we follow the so-called localisation method of $\Gamma$-convergence \cite{DM}, with the caveat that, since we prove a sequential $\Gamma$-convergence result, 
we cannot apply the abstract general theory directly.

As a first step, we localise the functionals, and define $F_{\varepsilon} : L^{1}(\Omega) \times \mathcal{A}(\Omega) \longrightarrow [0,+\infty]$ as 
\begin{equation}
F_{\varepsilon}(u,U) := 
\begin{cases} 
\displaystyle
\int_{U}|\nabla u|^{2} \dx + \mathcal{H}^{n-1}(S_{u}\cap(U\cap\varepsilon P)) + \beta_{\varepsilon}\mathcal{H}^{n-1}(S_{u}\cap (U \setminus \varepsilon P)) 
& \!\!\!\mbox{if } u \in GSBV^{2}(U), \\ 
+\infty 
&  \!\!\mbox{otherwise in } L^{1}(\Omega).
\end{cases}
\label{GG2}
\end{equation}

	\begin{rem}\label{T3}
For every fixed $\varepsilon>0$ the functional $F_{\varepsilon}: L^{1}(\Omega) \times \mathcal{A}(\Omega) \longrightarrow [0,+\infty]$ satisfies the following properties, for every $u\in L^{1}(\Omega)$, and $U\in \mathcal{A}(\Omega)$: 
	\begin{enumerate}
		\item $F_{\varepsilon}(u,\cdot)$ is increasing: 
		$F_{\varepsilon}(u,V)\leqslant F_{\varepsilon}(u,U)$
		for every $V\in\mathcal{A}(\Omega)$, $V \subset U$; 
		\item $F_{\varepsilon}(u,\cdot)$ is super-additive: for
		$U,V\in \mathcal{A}(\Omega)$, $U\cap V=\emptyset$, then  
		\[F_{\varepsilon}(u,U \cup V)\geqslant F_{\varepsilon}(u,U)
		+F_{\varepsilon}(u,V);\]
		\item $F_{\varepsilon}$ is local: $F_{\varepsilon}(u,U)=F_{\varepsilon}(v,U)$  
		for every $v\in L^{1}(\Omega)$ such that $u=v$ $\mathcal{L}^{n}$-a.e. in $U$;
		\item $F_{\varepsilon}$ is invariant by translations in $u$:
		$F_{\varepsilon}(u+s,U)=F_{\varepsilon}(u,U)$ for every $s\in\mathbb{R}$; 
		\item $F_{\varepsilon}(\cdot,U)$ is decreasing by truncations:
		$F_{\varepsilon}(u^m,U)\leqslant F_{\varepsilon}(u,U)$ for every $m>0$,
		where $u^m:=(u\wedge m)\vee(-m)$.
	\end{enumerate} 
\end{rem}

Let $(\e_k)$  be a positive sequence of real numbers decreasing to $0$. We define the localised versions $F', F'': L^1(\Om)\times \Am(\Om)\to [0,+\infty]$ of \eqref{def:G-li} and \eqref{def:G-ls} as 
\begin{equation}\label{Fpsloc}
F'(\cdot, U):=\Gamma\hbox{-}\liminf_{k\to+\infty} F_{\e_k}(\cdot, U),\qquad
F''(\cdot, U):=\Gamma\hbox{-}\limsup_{k\to+\infty} F_{\e_k}(\cdot, U),
\end{equation}
for every $U\in \Am(\Om)$. We notice that in our case the functionals $F'$ and $F''$ will depend on $\ell$, however to simplify the notation at this stage we prefer to omit this dependence. 

It is easy to prove that $F'$ and $F''$ are both lower semicontinuous with respect to the convergence in Definition \ref{def:conv}.
In view of Remark \ref{T3}, they are easily seen to be increasing and local; further, $F'$ is superadditive. 
Moreover, by combining $(5)$ of Remark \ref{T3} and iii. of Remark \ref{conv-deb-troncata} it can be immediately checked that both $F'$ and $F''$ decrease by truncation.

In general $F'(u,\cdot)$ and $F''(u,\cdot)$ are not inner regular. Therefore we also consider their inner regular envelopes, that is the functionals $F'_-,F''_-\colon L^1(\Om)\times \Am(\Om)\longrightarrow[0,+\infty]$ defined as
\begin{equation*}
{F}'_-(u,U):=\sup\big\{{F}'(u,V)\colon V\subset \subset U,\; V\in \Am(\Om)\big\}.
\end{equation*}
and
\begin{equation*}
{F}''_-(u,U):=\sup\big\{{F}''(u,V)\colon V\subset \subset U,\; V\in \Am(\Om)\big\}.
\end{equation*}

\begin{rem}\label{pro-inner-reg-env} 
The functionals ${F}'_-$ and ${F}''_-$ are both increasing, lower semicontinuous \cite[Remark 15.10]{DM}, and local \cite[Remark 15.25]{DM}, and decrease by truncations. Moreover, $F'_-$ is superadditive \cite[Remark 15.10]{DM}. 
\end{rem}

The next proposition is the analogue of the compactness result \cite[Theorem 16.9]{DM}, when the sequential notion of $\Gamma$-convergence in Definition \ref{def:G-conv} is taken into account.

\begin{propo} Let $F_{\varepsilon} \colon L^1(\Om) \times \Am(\Om) \longrightarrow [0,+\infty]$ be the functionals defined in \eqref{GG2}. 
Then for every sequence of positive numbers decreasing to zero, there exists a subsequence $(\e_k)$ such that the corresponding functionals $F'$ and $F''$ in \eqref{Fpsloc} 
	satisfy
	$F'_-=F''_-$. 
\end{propo}
\begin{proof}
Let $\mathcal{R}(\Omega)$ be the class of all finite unions of open rectangles of $\Omega$ with rational vertices. Let $R \in \mathcal{R}(\Omega)$ and $u\in L^1(R)$ be fixed. Since $F'$ in \eqref{Fpsloc} is actually attained along a sequence we can find $(u_{\varepsilon}) \subset L^{1}(R)$ converging to $u$ in the sense of Definition \ref{def:conv} such that
	\begin{eqnarray*}
		F''(u,R) \geqslant F'(u,R)
		&=& \liminf_{\varepsilon \rightarrow 0} F_{\varepsilon}(u_{\varepsilon},R) 
		= \lim_{k \rightarrow +\infty} F_{\varepsilon_{k}}(u_{\varepsilon_{k}},R) \\
		&=& \limsup_{k \rightarrow +\infty} F_{\varepsilon_{k}}(u_{\varepsilon_{k}},R) 
		\geqslant  F''(u,R),
	\end{eqnarray*}
where the subsequence $\e_k$ depends on $R$. Since $\mathcal{R}(\Omega)$ is countable, by a diagonal argument, we can find a subsequence $\tilde{\varepsilon}_{k}$ such that $F'(u,R) = F''(u,R)$ for every $u \in L^{1}(\Om)$ and for every $R \in \mathcal{R}(\Om)$. 
Since $\mathcal{R}(\Omega)$ is dense in $ \mathcal{A}(\Omega)$, we have
	\begin{align*}
		{F}'_-(u,U) 
		&= \sup\{{F}'(u,V):  V \in \mathcal{A}(\Omega), \ V \subset\subset U\} 
		= \sup\{{F}'(u,R):  R \in \mathcal{R}(\Om), \ R \subset\subset U\} \\
		&= \sup\{{F}''(u,R):  R \in \mathcal{R}(\Om), \ R \subset\subset U\} 
		= \sup\{{F}'(u,V): V \in \mathcal{A}(\Omega), \ V \subset\subset U\} \\		
		&= {F}''_-(u,U),
	\end{align*}
which concludes the proof.
\end{proof}
We now reintroduce the $\ell$-dependence in our notation and define the functional 
\begin{equation}\label{Def:astratta-F}
F^\ell_{\rm{hom}}:=F'_-=F''_-. 
\end{equation}
We observe that by monotonicity we always have
\begin{equation}\label{trivial}
F''_-=F'_-\leqslant F'\leqslant F''.
\end{equation}
Therefore, if we show that $F''$ is inner regular\ie $F''=F''_-$, by combining \eqref{Def:astratta-F} and \eqref{trivial} we get $F'=F''=F^\ell_{\rm{hom}}$, and therefore the $\Gamma$-convergence of the subsequence $(F_{\e_{k}})$ to the functional $F^\ell_{\rm{hom}}$.

A first step towards the proof of the inner-regularity of $F''$ consists in proving that the functionals $F_{\varepsilon}$ satisfy a fundamental estimate, uniformly in $\e$. The fundamental estimate we prove is non-standard. Indeed, we 
need an error term that is infinitesimal for the convergence in Definition \ref{def:conv}, and hence that only weights the functions outside the weak inclusion. We note that an analogous estimate was also established by Braides and Garroni in \cite[Proposition 3.3]{BrGa}.

\begin{lemma}(Fundamental Estimate in perforated domains).\label{lemma:FE} 
For every $\eta>0$, and for every $U'$, $U''$, $V \in \mathcal{A}(\Omega)$, with $U'\subset\subset U''$, there exists a constant $M(\eta)>0$ satisfying the following property: for every $\varepsilon>0$, for every $u \in L^{1}(\Omega)$ with $u \in SBV^{2}(U'')$, and for every $v \in L^{1}(\Omega)$ with $v \in SBV^{2}(V)$, there exists a function $\varphi \in {C}_{0}^{\infty}(\Omega)$ with $\varphi = 1$ in a neighbourhood of $U'$, $\supp\, \varphi \subset U''$ and $0 \leqslant \varphi \leqslant 1$ such that
		\begin{equation}
			F_{\varepsilon}
			(\varphi u + (1-\varphi) v, U' \cup V)
			\leqslant (1+\eta)\left(F_{\varepsilon}(u,U'') 
			+ F_{\varepsilon}(v,V)\right)
			+ M(\eta)||u-v||_{L^{2}(S\cap \e P)}
			\label{FE}
		\end{equation}
with $S:=(U''\setminus U') \cap V$.
\end{lemma}

\begin{proof}
Let $0<\eta<1$, $U'$, $U''$, and $V\in\mathcal{A}(\Omega)$ be as in the statement. Let $\delta>0$ be small enough so that $Q_{\frac{1}2+\delta}\subset \subset Q$, and let $\psi\in C_0^\infty(Q)$ be a cut-off function between $Q_{\frac12}$ and $Q_{\frac{1}2+\delta}$, namely a function such that $0\leqslant \psi\leqslant 1$, $\psi \equiv 1$ on $Q_{\frac{1}{2}}$, and $\supp \psi\subset  Q_{\frac{1}2+\delta}$.  

For $\e>0$ and $i\in \Z^n$, we define the operator $R_{i}^{\varepsilon} : W_{\textrm{loc}}^{1,\infty}(\mathbb{R}^{n}) \rightarrow W_{\textrm{loc}}^{1,\infty}(\mathbb{R}^{n})$ as
	\begin{equation*}
		R_{i}^{\varepsilon}(\phi)(x) 
		:= \left(1-\psi\left(\dfrac{x}{\varepsilon} - i \right)\right)\phi(x)
		+ \psi\left(\dfrac{x}{\varepsilon} - i \right)\avint_{\varepsilon i +\varepsilon Q_{\frac{1}2+\delta}}
		\phi(y) \ dy.
	\end{equation*}
It is easy to see that $R_{i}^{\varepsilon}(\phi)(x) = \phi(x)$ if $x\notin \varepsilon i +\varepsilon Q_{\frac{1}2+\delta}$, and $R_{i}^{\varepsilon}(w)$ is instead constant in $\varepsilon i +\varepsilon Q_{\frac{1}2}$. 
Finally, we define the operator $\mathcal R^{\varepsilon} : W_{\textrm{loc}}^{1,\infty}(\mathbb{R}^{n}) \rightarrow W_{\textrm{loc}}^{1,\infty}(\mathbb{R}^{n})$ as
	\begin{equation*}
		\mathcal R^{\varepsilon}(\phi)(x):=
		\begin{cases}
		\smallskip
		R_{i}^{\varepsilon}(\phi)(x) & \quad \textrm{if } x\in \varepsilon i +\varepsilon Q_{\frac{1}2+\delta},\; i\in \Z^n,\\
		\phi(x) & \quad \textrm{otherwise}.
		\end{cases}
	\end{equation*}
Let $W',W'' \in \mathcal{A}(\Omega)$ be such that $U' \subset\subset W' \subset\subset W'' \subset\subset U''$.
Let $\phi$ be a cut-off function between the sets $W'$ and $W''$ and set $\varphi:= \mathcal{R}^{\varepsilon}(\phi)$; then $\varphi$ is a cut-off between the sets $U'$ and $U''$, provided $\delta$ is small enough. Moreover, by construction, $D \varphi = 0$ on $\Omega\setminus \e P$ and 
$\|D \varphi\|_{L^\infty(\Om)} \leq C \|D \phi\|_{L^\infty(\Om)}$, with a uniform 
constant independent of $\e$ (see e.g., \cite[Remark 2.7]{BrGa}).

Let $u$ and $v$ be as in the statement and let $w:= \varphi u + (1-\varphi)v$; clearly, $w\in SBV^{2} (U' \cup V)$. Then
\begin{equation}
		F_{\varepsilon}(w,U' \cup V)
		= F_{\varepsilon}(u,U')
		+ F_{\varepsilon}^{*}(v,V \setminus U'')
		+ F_{\varepsilon}^{*}(w,S),
		\label{FE1}
	\end{equation}
where, for fixed $w\in L^1(\Omega)$, $F_{\varepsilon}^{*}(w,\cdot)$ denotes the measure that extends $F_{\varepsilon}(w,\cdot)$ to the 
$\sigma$-algebra $\mathcal{B}(\Omega)$ of Borel subsets of $\Omega$, and is defined as
	\begin{equation*}
		F_{\varepsilon}^{*}(w,B) 
		:= \inf\{ F_\e(w, U): U\in \mathcal{A}(\Omega), B\subset U\},
	\end{equation*}
for every $B \in \mathcal{B}(\Omega)$. Now we estimate the last term in the right-hand side of \eqref{FE1}. For any fixed $\eta \in (0,1)$ by convexity and by the definition of $\varphi$ we have
	\begin{align*}
		F_{\varepsilon}^{*}(w,S)
		&\leqslant \int_{S \cap \varepsilon P} \Big|(1-\eta)\dfrac{\varphi\nabla u + (1-\varphi)\nabla v}{1-\eta} + \eta \dfrac{D \varphi(u-v)}{\eta}\Big|^{2} \ dx \\
		&+ \int_{S \setminus \varepsilon P} \Big|(1-\eta)\dfrac{\varphi\nabla u + (1-\varphi)\nabla v}{1-\eta}\Big|^{2} \ dx \\
		&+ \mathcal{H}^{n-1}(S_{u} \cap(S\cap \e P)) + \mathcal{H}^{n-1}(S_{v} \cap (S\cap \e P)) \\
		&+\beta_\e \mathcal{H}^{n-1}(S_{u} \cap (S\setminus \e P)) + \beta_\e\mathcal{H}^{n-1}(S_{v} \cap (S\setminus \e P)) \\
		&\leqslant \dfrac{1}{1-\eta}\left(\int_{S} |\nabla u|^{2} \ dx + \int_{S} |\nabla v|^{2} \dx\right) + \dfrac{1}{\eta}\int_{S \cap \varepsilon P} |D \varphi|^{2}|u-v|^{2} \dx \\
		&+ \mathcal{H}^{n-1}(S_{u} \cap(S\cap \e P)) + \mathcal{H}^{n-1}(S_{v} \cap (S\cap \e P)) \\
		&+\beta_\e \mathcal{H}^{n-1}(S_{u} \cap (S\setminus \e P)) + \beta_\e\mathcal{H}^{n-1}(S_{v} \cap (S\setminus \e P)) \\
		&\leqslant\dfrac{1}{1-\eta}\left(F_{\varepsilon}^{*}(u,S) + (F_{\varepsilon}^{*}(v,S)\right) + \dfrac{1}{\eta} \|D\varphi\|^2_{L^\infty(\Om)}\int_{S\cap \varepsilon P} |u - v|^{2} \ dx\\
		&\leqslant\dfrac{1}{1-\eta}\left(F_{\varepsilon}^{*}(u,S) + (F_{\varepsilon}^{*}(v,S)\right) + \dfrac{C}{\eta} \int_{S\cap \varepsilon P} |u - v|^{2} \ dx.
	\end{align*}
This, together with  \eqref{FE1}, concludes the proof.
\end{proof}

\begin{theo}[$\Gamma$-convergence and properties of the $\Gamma$-limit]\label{Gamma-compactness}
Let $(\e_k)$ be the sequence for which $F'_-=F''_-$. Then 
$F^\ell_{\rm{hom}}$ defined in \eqref{Def:astratta-F} satisfies the following properties:

\begin{itemize}

\item [i)] {\rm (locality and lower semicontinuity)} for every $U\in\Am(\Om)$, the functional $F^\ell_{\rm{hom}}(\cdot,U)$ is local and lower semicontinuous with respect to the strong $L^1(\Om)$-topology; 

\item [ii)]  {\rm (measure property)} for every $u\in GSBV^2(\Om)$, the set function $F^\ell_{\rm{hom}}(u,\cdot)$ is the restriction to $\Am(\Om)$ of a Radon measure on $\Om$;

\item [iii)]  {\rm ($\Gamma$-convergence)} for every $U\in\Am(\Om)$
$$
F^\ell_{\rm{hom}}(\cdot, U)=F'(\cdot,U)=F''(\cdot, U)\quad \text{on\; $GSBV^2(\Om)$;}
$$

\item [iv)]  {\rm (translational invariance in $u$)} 
		for every $u \in L^{1}(\Omega)$ and  $U \in \mathcal{A}(\Omega)$
			\begin{equation*}
				F^\ell_{\rm{hom}}(u+s,U) = F^\ell_{\rm{hom}}(u,U) \text{ for every } s \in \mathbb{R};
			\end{equation*}	
\item [v)]  {\rm (translational invariance in $x$)} 
		for every $u \in L^{1}(\Omega)$ and  $U \in \mathcal{A}(\Omega)$
			\begin{equation*}
				F^\ell_{\rm{hom}}(u(\cdot-y),U + y) = F^\ell_{\rm{hom}}(u,U) \text{ for every } y \in \mathbb{R}^{n} \text{ such that } U+y \subset\subset \Omega.
			\end{equation*}

		\end{itemize}
\end{theo}

\begin{proof}
Property i) is a straightforward consequence of Remark \ref{pro-inner-reg-env}, since lower semicontinuity with respect to the convergence in Definition \ref{def:conv} implies lower semicontinuity in $L^1(\Om)$. Property ii) follows by the measure-property criterion of the De Giorgi and Letta (see \cite[Theorem 14.23]{DM}) once we show that for every $u\in GSBV^2(\Om)$ the set function $F^\ell_{\rm{hom}}(u,\cdot)$ is subadditive. The proof of the subadditivity of $F^\ell_{\rm{hom}}(u,\cdot)$ follows from the fundamental estimate Lemma \ref{lemma:FE}. In our case the main difference with respect to a standard situation is that the reminder in \eqref{FE} is given in terms of $\|u-v\|_{L^2(S\cap \e P)}$, while we are studying the $\Gamma$-convergence of $F_\e$ with respect to the convergence in Definition \ref{def:conv} which only ensures that $\|u-v\|_{L^1(S\cap \e P)}$ tends to zero. Therefore we provide a detailed proof of the subadditivity of $F^\ell_{\rm{hom}}(u,\cdot)$.     

We start observing that on $GSBV^2(\Om)\cap L^\infty(\Om)$ the equality $F'_-=F''_-$ in particular implies that the following limsup-type inequality is satisfied: for every $u\in GSBV^2(\Om)\cap L^\infty(\Om)$ and for every $U,U'\in\Am(\Om)$ with $U'\subset\subset U$, there exists a sequence $(u_k)\subset GSBV^2(U')\cap L^1(\Om)$ with $u_k\to u$ in the sense of Definition \ref{def:conv} such that
$$
\limsup_{k\to +\infty}F_{\e_k}(u_k,U') \leqslant F^\ell_{\rm{hom}}(u,U)
$$
(see e.g. \cite[Proposition 16.4 and Remark 16.5]{DM}, since both the infima in the definition of $F'$ and $F''$ are actually minima).

Now let $U,V\in\Am(\Om)$ and let $u\in GSBV^2(\Om)\cap L^\infty(\Om)$. Fix any $U'\subset\subset U$, $V'\subset\subset V$, $U',V'\in \Am(\Om)$. Choose an open set $U''$ such that $U'\subset\subset U''\subset\subset U$ and two sequences $(u_k)\subset GSBV^2(U'')\cap L^1(\Om)$ and $(v_k)\subset GSBV^2(V')\cap L^1(\Om)$, with $u_k\to u$ and $v_k\to u$ in the sense of Definition \ref{def:conv}, and such that
$$
\limsup_{k\to +\infty}F_{\e_k}(u_k,U'')\leqslant F^\ell_{\rm{hom}}(u,U), \quad \limsup_{k\to +\infty}F_{\e_k}(v_k,V')\leqslant F^\ell_{\rm{hom}}(u,V).
$$
Moreover, since the functionals $F_{\e_k}$ decrease by truncation, we can additionally assume that $\|\tilde u_k\|_{L^\infty(\Om)}\leqslant \|u\|_{L^\infty(\Om)}$, $\|\tilde v_k\|_{L^\infty(\Om)}\leqslant \|u\|_{L^\infty(\Om)}$, where $\tilde u_k$ and $\tilde v_k$ are the extended sequences in Definition \ref{def:conv}. In particular $\tilde u_k\to u$ in $L^2(\Om)$ and $\tilde v_k\to u$ in $L^2(\Om)$.

Let $\eta>0$ be arbitrary but fixed. Then, the fundamental estimate Lemma \ref{lemma:FE} provided us with a constant $M(\eta)>0$ and a sequence $(\varphi_k)$ of cut-off functions between $U'$ and $U''$ such that
\begin{eqnarray*}
&F_{\e_k}(\varphi_k u_k+(1-\varphi_k)v_k, U'\cup V')\\& \leqslant (1+\eta)\Big(F_\e(u_k,U'')+F_{\e_k}(v_k,V')\Big)+M(\eta)\|u_k-v_k\|^2_{L^2(\Om\cap \e P)}.
\end{eqnarray*}
Hence, taking the limit as $k\to+\infty$, and noticing that $\varphi_k u_k+(1-\varphi_k)v_k\to u$ in the sense of Definition \ref{def:conv}, we get

$$
F^\ell_{\rm{hom}}(u,U'\cup V')\leqslant (1+\eta)\Big(F^\ell_{\rm{hom}}(u,U)+F^\ell_{\rm{hom}}(u,V)\Big).
$$
Now letting $\eta\to 0$, and then $U'\nearrow U$, $V'\nearrow V$ in view of the inner-regularity of $F^\ell_{\rm{hom}}$ we get
$$
F^\ell_{\rm{hom}}(u,U\cup V)\leqslant F^\ell_{\rm{hom}}(u,U)+F^\ell_{\rm{hom}}(u,V),
$$
hence the subadditivity of $F^\ell_{\rm{hom}}(u,\cdot)$ for $u\in GSBV^2(\Om)\cap L^\infty(\Om)$.

Now let $u\in GSBV^2(\Om)$ and, for every $m\in\N$, set $u^m:=(u\wedge m)\vee (-m)$. Since $u^m\in L^\infty(\Om)$ and $F^\ell_{\rm{hom}}$ decreases by truncation (see Remark \ref{pro-inner-reg-env}), we have
$$
F^\ell_{\rm{hom}}(u^m, U\cup V)\leqslant F^\ell_{\rm{hom}}(u^m, U)+F^\ell_{\rm{hom}}(u^m, V) \leqslant F^\ell_{\rm{hom}}(u, U)+F^\ell_{\rm{hom}}(u, V).
$$
On the other hand, since $u^m\to u$ in $L^1(\Om)$, the lower semicontinuity of $F^\ell_{\rm{hom}}$ yields 
$$
F^\ell_{\rm{hom}}(u, U\cup V)\leqslant \liminf_{m\to +\infty}F^\ell_{\rm{hom}}(u^m, U\cup V)\leqslant F^\ell_{\rm{hom}}(u, U)+F^\ell_{\rm{hom}}(u, V),
$$
thus the subadditivity of $F^\ell_{\rm{hom}}(u,\cdot)$ for every $u\in GSBV^2(\Om)$.

We now turn to the proof of iii). This will be achieved by showing that $F''$ is inner-regular. As a consequence we will have $F''=F''_-=F'_-=F^\ell_{\rm{hom}}\leqslant F'$ and therefore the $\Gamma$-convergence of $F_{\e_k}$ to $F^\ell_{\rm{hom}}$. 

The inner regularity of $F''$ is again a consequence of the fundamental estimate Lemma \ref{lemma:FE}. In fact, let $MS$ be as in \eqref{MS1} and fix $W\in \Am(\Om)$. Let $u\in GSBV^2(\Om)$; since $MS(u,\cdot)$ is (the restriction of) a Radon measure on $\Am(\Om)$, for every $\eta>0$ there exists a compact set $K\subset W$ such that and $MS(u, W\setminus K)<\eta$.

Now choose $U, U'\in \Am(\Om)$ satisfying 
$K\subset U'\subset \subset U \subset \subset W$
and set $V:=W\setminus K$. Recalling that $F''(u, \cdot)$ is increasing, Lemma \ref{lemma:FE} easily yields
\begin{eqnarray*}
F''(u, W) \leqslant F''(u, U'\cup V)\leqslant F''(u,U)+F''(u,V) =F''(u,U)+ F''(u,W\setminus K).
\end{eqnarray*}
Moreover, by the definition of $F''_-$ and by the trivial estimate $F_\e \leqslant MS$ we have
\begin{eqnarray*}
F''(u, W) \leqslant F''_-(u,W)+MS(u,W\setminus K)
\leqslant F''_-(u,W)+ \eta.
\end{eqnarray*}
Hence by the arbitrariness of $\eta$ we get 
$$
F''(u, W) \leqslant F''_-(u,W)\quad \text{for every}\; W\in\Am(\Om), \; u\in GSBV^2(\Om).
$$
Therefore, as the opposite inequality is trivial, we deduce that $F''(u,\cdot)$ is inner regular. 

Finally, the proof of iv) and v) is standard and follows as in e.g. \cite[Lemma 3.7]{BDfV}. 
\end{proof}

\subsection{Integral representation of the $\Gamma$-limit} In this subsection we show that the $\Gamma$-limit $F^\ell_{\rm{hom}}$ 
can be represented in an integral form.

\begin{theo}[Integral representation]\label{t:int-rep}
Let $F^\ell_{\rm{hom}}$ be the functional whose existence is established in Theorem \ref{Gamma-compactness}. Then, for every $u\in GSBV^2(\Om)$ and every $U\in \Am(\Om)$ we have
\begin{equation}\label{s:int-rep}
F^\ell_{\rm{hom}}(u,U)=\int_U f^\ell_{\rm{hom}}(\nabla u)\dx +\int_{S_u\cap U}g^\ell_{\rm{hom}}([u],\nu_u)\, d\mathcal H^{n-1} 
\end{equation}
for a convex function $f^\ell_{\rm{hom}} \colon \R^n \longrightarrow [0,+\infty)$ satisfying for every $\ell \in [0,+\infty]$ and every $\xi \in \R^n$ the following bounds
\begin{equation}\label{c:bd-on-fhom}
\hat f(\xi) \leqslant f^\ell_{\rm{hom}}(\xi) \leqslant |\xi|^2
\end{equation}
and a Borel function $g^\ell_{\rm{hom}} \colon \R\times \Sph^{n-1} \longrightarrow [0,+\infty)$ satisfying:

 i) {\rm (monotonicity in $t$ and symmetry)} for any fixed $\nu\in \Sph^{n-1}$, $g^\ell_{\rm{hom}} (\cdot,\nu)$ is nondecreasing on $(0,+\infty)$ and satisfies the symmetry condition $g^\ell_{\rm{hom}}(-t,-\nu)=g^\ell_{\rm{hom}}(t,\nu)$ for $t\in \R$;

ii) {\rm (subadditivity in $t$)} for any $\nu\in \Sph^{n-1}$
$$
g^\ell_{\rm{hom}}(t_1+t_2,\nu)\leqslant g^\ell_{\rm{hom}}(t_1,\nu)+ g^\ell_{\rm{hom}}(t_2,\nu), 
$$
for every $t_1,t_2\in\R$; 

iii)  {\rm (convexity in $\nu$)} for any $t\in\R$, the $1$-homogeneous extension of $g^\ell_{\rm{hom}}(t,\cdot)\colon \Sph^{n-1}\longrightarrow [0,+\infty)$ to $\R^n$ is convex. This condition can be also equivalently expressed in terms of the function $g^\ell_{\rm{hom}}$ as
$$
g^\ell_{\rm{hom}}(t,\nu)\leqslant \lambda_1 g^\ell_{\rm{hom}}(t,\nu_1)+
\lambda_2 g^\ell_{\rm{hom}}(t,\nu_2),
$$
for every $\nu,\nu_1,\nu_2\in \Sph^{n-1}$, $\lambda_1,\lambda_2\geqslant 0$ such that $\lambda_1 \nu_1+\lambda_2 \nu_2=\nu$.
 \end{theo}

\begin{proof}
We recall that $F_\e$ satisfies the bound \eqref{BB1}, namely 
$
\widehat F_\e \leqslant F_\e \leqslant MS
$
in $L^1(\Om)$. 
Therefore by $\Gamma$-convergence and in view of Remark \ref{rem:convBS}  and Remark \ref{conv-deb-troncata} ii.  we get
\begin{equation}\label{c:bd-mancante}
\widehat F \leqslant F^\ell_{\rm{hom}}\leqslant MS \quad \textrm{in } L^1(\Om),
\end{equation}
thus invoking $\eqref{BB2}$ we deduce 
\begin{equation}\label{bounds-MS-type}
\min\{c_1,c_2\}\,MS \leqslant F^\ell_{\rm{hom}}\leqslant MS \quad \textrm{in } L^1(\Om),
\end{equation}
where $c_1$ and $c_2$ are as in \eqref{B1} and \eqref{B2}, respectively. 
Hence, Theorem \ref{Gamma-compactness}, \eqref{bounds-MS-type}, \cite[Theorem 1]{BFLM}, and a standard perturbation argument (see e.g. \cite[Theorem 2.2]{DMZ}) yield the integral representation \eqref{s:int-rep} on the space $SBV^2(\Om)$. Then a standard truncation and continuity argument allows to extend this integral representation to the whole space $GSBV^2(\Om)$ and thus to get exactly \eqref{s:int-rep}. 

The convexity of $f^\ell_{\rm{hom}}$, the subadditivity of $g^\ell_{\rm{hom}}$ in $t$ and the convexity in $\nu$ of its $1$-homogeneous extension to $\R^n$ are straightforward consequences of the $L^1(\Om)$-lower semicontinuity of $F^\ell_{\rm{hom}}$. 
Since moreover $f^\ell_{\rm hom}(\xi)=F^\ell_{\rm hom}(u_\xi, Q)$ where $u_\xi(x):=\xi \cdot x$, the bounds in \eqref{c:bd-on-fhom} are an immediate consequence of \eqref{c:bd-mancante}. 
Finally, the monotonicity in $t$ and the symmetry of $g^\ell_{\rm{hom}}$ follow from \cite[Theorem 1]{BFLM}. 
\end{proof}


\section{Identification of the homogenized  surface integrand}\label{sect:g}
 In this section we identify the limit surface integrand $g^\ell_{\rm{hom}}$. To do so we make use of the following technical lemma (see \cite[Lemma 4.5]{TC1}, and see also \cite[Lemma 2.5]{TC2} for a more general version of the result).

\begin{lemma}[``Fracture Lemma'']
	Let $n \geqslant 2$ and $\eta \in (0,1]$ be fixed. There exists a constant $\gamma=\gamma(n,\eta)>0$ such that if $0<s\leqslant r$, and $u\in \mathcal{P}(B_{r,r+s})$ verify the following hypotheses: 
	\begin{enumerate}
		\item[$(H1)$]$\displaystyle\mathcal{H}^{n-1}(S_{u} \cap B_{r,r+s}) 
		\leqslant \mathcal{H}^{n-1}(S_{v} \cap B_{r,r+s})$ 
		for every competitor $v \in \mathcal{P}(B_{r,r+s})$ satisfying  $\supp(u-v) \subset B_{r,r+s}$; 
		\smallskip
		\item[$(H2)$]$\displaystyle \mathcal{H}^{n-1}(S_{u} \cap B_{r,r+s}) \leqslant \gamma s^{n-1}$; 
	\end{enumerate}
	then for every $r_{0}$ and $s_{0}$ such that $r \leqslant r_{0} < r_{0} + s_{0} \leqslant r+s$ and $s_{0} \geqslant \eta s$, there exists a radius $\bar{r} \in (r + s_{0}/3, r_{0} + 2 s_{0}/3)$ with the property that
	\begin{equation*}
	S_{u} \cap \partial B_{\bar{r}} = \emptyset.
	\end{equation*}
	\label{fracture-lemma}
\end{lemma}

The following theorem is the main result of this section. 

\begin{theo}[Identification of the homogenized  surface integrand]\label{t:g-hom-g-hat}
Let $g^\ell_{\rm hom}$ be the function as in Theorem \ref{t:int-rep} and let $\hat g$ be as in \eqref{G}. Then, for every $(t,\nu) \in (\R\setminus\{0\})\times \Sph^{n-1}$  and every $\ell \in [0,+\infty]$ we have $g^\ell_{{\rm hom}}(t,\nu)=\hat g(\nu)$. 
\end{theo}

\begin{proof}
For $\e,\delta>0$ we define the comparison functional $G_{\varepsilon}^{\delta}: L^1(\Om) \longrightarrow [0,+\infty]$ as
			\begin{equation*}
				G_{\varepsilon}^{\delta}(u) 
				:=  
				\begin{cases} 
					\displaystyle
					\int_{\Omega}|\nabla u|^{2}  dx 
					+ \mathcal{H}^{n-1}(S_{u} \cap \varepsilon P)
					+ \delta \int_{S_{u}}
					(1+|[u]|) \ d\mathcal{H}^{n-1}
					& \mbox{if} \ u \in SBV^{2}(\Omega), \\ 
					+\infty
					& \mbox{otherwise in } L^{1}(\Omega). 
				\end{cases}
			\end{equation*}	
Let $\delta>0$ be fixed; by \cite[Theorem 2.3]{BDfV} we deduce that, as $\e$ tends to zero, the functionals $G_\e^\delta$ $\Gamma$-converge, with respect to the $L^1(\Om)$-convergence, to the functional $G^{\delta}: L^{1}(\Omega) \rightarrow [0,+\infty]$ given by		
\begin{equation*}
			G^{\delta}(u) 
			:=  
			\begin{cases} 
			\displaystyle
			\int_{\Omega}|\nabla u|^2 \dx + 
			\int_{S_{u}}
			g^{\delta}([u],\nu_{u}) \ d\mathcal{H}^{n-1},
			& \mbox{if} \ u \in SBV^{2}(\Omega), \\ 
			+\infty,
			& \mbox{otherwise in } L^{1}(\Omega), 
			\end{cases}
		\end{equation*}	
	where $g^{\delta}: \R\times \Sph^{n-1} \to [0,+\infty]$ is defined as
			\begin{align}\label{g-delta-t}
			g^{\delta}(z,\nu) 
			:= \lim_{t \rightarrow + \infty} \dfrac{1}{t^{n-1}}
			\inf \bigg\{ 
			\int_{S_{u}\cap tQ^{\nu}} \big(
			\chi_{P}
			+ \delta (1+|[u]|)\big)
			d\mathcal{H}^{n-1} : \ 
			u \in SBV^{\rm pc}(tQ^{\nu}), \nonumber \\ 
			  u = u_{0,z}^\nu \text{ in a neighbourhood of } \partial (tQ^{\nu})
			\bigg\}.
		\end{align}
Let $\varepsilon>0$ be small enough to have that $\beta_{\varepsilon} < \delta$; then we immediately deduce the bounds
		\begin{equation}
			\widehat{F}_{\varepsilon}(u)
			\leqslant F_{\varepsilon}(u)
			\leqslant G_{\varepsilon}^{\delta}(u), \quad \text{for every } u\in L^{1}(\Omega).
			\label{G1}
		\end{equation}
		From \eqref{G1} it follows that for every $z\in \R\setminus\{0\}$,  $\nu \in \Sph^{n-1}$, and  $\ell\in [0,+\infty]$ 
	\begin{equation}
			\hat{g}(\nu) \leqslant g^\ell_{\hom}(z,\nu) \leqslant g^{\delta}(z,\nu).
			\label{GGG1}
		\end{equation}	
Indeed let $(u_\e)\subset L^1(\Om)$ be a recovery sequence for $G_\e^\delta$, with $u_\e \to u_{0,z}^\nu$ strongly in $L^1(\Om)$. Then, since 
$u_\e \to u_{0,z}^\nu$ in the sense of Definition \ref{def:conv}, we have 
\begin{align*}
g^\delta(z,\nu) \mathcal{H}^{n-1}(\Om\cap\Pi^\nu_0) &= G^\delta(u_{0,z}^\nu) = \lim_{\e\to 0}G_\e^\delta(u_\e)\geqslant \liminf_{\e\to 0}F_\e(u_\e) \\
&\geqslant F^\ell(u_{0,z}^\nu) = g^\ell_{\hom}(z,\nu) \mathcal{H}^{n-1}(\Om\cap\Pi^\nu_0).
\end{align*}
Similarly, let $(u_\e)\subset L^1(\Om)$ be a recovery sequence for $F_\e$, with $u_\e \to u_{0,z}^\nu$  in the sense of Definition \ref{def:conv}. Then there exists a sequence $\tilde u_\e$ with $\tilde u_\e = u_\e$ in $\Omega\cap \e P$ and such that $\tilde u_\e$ converges to $u_{0,z}^\nu$ strongly in $L^1(\Om)$; hence 
\begin{align*}
g^\ell_{\hom}(z,\nu) \mathcal{H}^{n-1}(\Om\cap\Pi^\nu_0) &= F^\ell(u_{0,z}^\nu) = \lim_{\e\to 0}F_\e(u_\e)\geqslant \liminf_{\e\to 0}\widehat F_\e(u_\e) = \liminf_{\e\to 0}\widehat F_\e(\tilde u_\e) \\
&\geqslant \widehat F(u_{0,z}^\nu) = \hat g(z,\nu) \mathcal{H}^{n-1}(\Om\cap\Pi^\nu_0),
\end{align*}
where $\Pi^{\nu}_0:= \{y\in \R^n: y\cdot \nu = 0\}.$

		We now claim that for every $z\in \R\setminus\{0\}$,  $\nu \in \Sph^{n-1}$
		\begin{equation}\label{claim-gh}
		g^{\delta}(z,\nu) \leqslant \hat{g}(\nu) + o(1),
		\end{equation}
	as $\delta \to 0$;  hence \eqref{claim-gh} together with \eqref{GGG1} will imply the thesis. 
	
	\smallskip

	To prove \eqref{claim-gh} let $\nu \in \Sph^{n-1}$ and $t>0$ be fixed, and let $\bar u \in \mathcal P(tQ^{\nu} \cap P)$ be such that $\bar u = u^{\nu}_{0,1}$ in a neighbourhood of $\partial (tQ^{\nu})$ and
	\begin{equation}\label{quasi-min-gh}
	\mathcal H^{n-1}  ({S_{\bar u}\cap tQ^{\nu} \cap P}) = \min \Big\{\mathcal H^{n-1} ({S_{ u}\cap tQ^{\nu} \cap P})\colon
			u \in \mathcal P(tQ^{\nu}\cap P),
			u = u_{0,1}^\nu \; \text{on}\;  \partial (tQ^{\nu})
			\Big\}.
	\end{equation}
	Note that the minimiser in \eqref{quasi-min-gh} exists: indeed any minimising sequence is weakly convergent in $BV(tQ^{\nu} \cap P)$, and the Hausdorff measure of the jump set in $tQ^{\nu} \cap P$ is lower semicontinuous.

	We now modify $\bar u$ in order to get a competitor for the minimisation problem in the definition of $g^{\delta}$.
	More precisely, we construct from $\bar u$ a Caccioppoli partition $w$ defined on the whole $tQ^{\nu}$ and such that 
	\begin{equation}\label{claim-d-z}
	\mathcal H^{n-1} (S_w\cap tQ^\nu \cap P)+\delta \int_{S_w\cap tQ^\nu}(1+|[w]|)\,d\mathcal H^{n-1}  \leqslant (1+\delta c(z))\mathcal H^{n-1}  (S_{\bar u}\cap tQ^\nu \cap P)+ \delta c(z) t^{n-1}
	\end{equation}
	holds true for some constant $c(z)>0$ independent of $t$ and $\delta$.
	
	It is convenient to write
		\begin{equation*}
			tQ^{\nu} \cap P 
			= \bigcup\limits_{k \in  \mathcal I^1}(Q^{k} \cap P) 
			\cup \bigcup\limits_{k\in \mathcal I^2}\big(Q^{k} \cap tQ^{\nu} \cap P\big),
		\end{equation*} 
	where $Q^k := Q+ k$, $\mathcal I^1:=\{k\in \Z^n \colon Q^k \subset tQ^{\nu}\}$, and $\mathcal I^2:=\{k \in \Z^n \colon Q^k \cap  \partial(tQ^{\nu})\neq\emptyset\}$.
	 We now illustrate in detail the multi-step construction of the desired function $w$.
	
	\medskip

\noindent \textit{Step 1: Modification of $\bar u$ in the ``internal" cubes $Q^k$, for $k\in \mathcal I^1$.} 
Let $k \in \mathcal I^1$, and set $u^k:= \bar u_{|Q^k\cap P}$. We recall that for our choice of $P$ we have $Q^k\cap P= Q^k_{1/2,1}$, where we defined $Q^{k}_{a,b}:=Q_{a,b}+ k$ for any $0<a<b\leqslant 1$. We also define $B^{k}_{a,b}:=B_{a,b}+ k$, for any $0<a<b\leqslant 1$.

\smallskip

Let $\sqrt{2}/4\leqslant r< r+ s<1$ be fixed.  We extend the function $u^k$ differently depending on whether the function $u^k$ admits a \emph{small} or a \emph{large} jump set in $B^{k}_{r,r+s}$ (note that $B^{k}_{r,r+s}\subseteq Q^k_{1/2,1}$).
More precisely, we say that ${u}^k$ has a small jump set in $B^{k}_{r,r+s}$ if
		\begin{equation}
			\mathcal{H}^{n-1}(S_{{u}^k} \cap B^{k}_{r,r+s}) \leqslant \gamma s^{n-1},
			\label{GC1}
		\end{equation}
where $\gamma=\gamma(n,\eta)$ is as in Lemma \ref{fracture-lemma}, corresponding to $\eta=1$.

Note that for the cubes with small jump, the assumptions $(H1)$ and $(H2)$ of Lemma \ref{fracture-lemma} are satisfied by $u^k$. Indeed, $(H1)$ follows by the minimality of $\bar u$ (defined in \eqref{quasi-min-gh}) in $t Q^\nu\cap P$, which implies its local minimality in every subset, and hence in particular the local minimality of $u^k$ in $B^{k}_{r,r+s}$, for every $k\in \mathcal I^1$. Finally, $(H2)$ is exactly \eqref{GC1}. 

By Lemma \ref{fracture-lemma} there exists $\bar{r} \in (r + s/3, r + 2 s/3)$ such that $S_{u^k} \cap \partial B^k_{\bar{r}} = \emptyset$, namely the trace of ${u}^k$ on $\partial B^{k}_{\bar{r}}$ is constant. We denote this constant value by $m^k$. Then we define the function $v^k$ on the whole $Q^k$ as follows
		\begin{equation}\label{vk-0}
			v^k :=  
			\begin{cases} 
			\displaystyle
			{u}^k
			& \mbox{in } Q^k \setminus \bar{B}^k_{\bar{r}}, \\
			m^k
			& \mbox{in } \bar{B}^k_{\bar{r}}. 
			\end{cases}
		\end{equation}	
	Clearly $v^k\in \mathcal P(Q^{k})$, and
			\begin{align*}
			\mathcal{H}^{n-1}(S_{v^k} \cap Q^{k})
			= \mathcal{H}^{n-1}(S_{{u}^k} \cap (Q^k \setminus \bar{B}^k_{\bar{r}}))
			\leqslant \mathcal{H}^{n-1}(S_{\bar u} \cap Q^{k} \cap P).
	\end{align*}
	Hence in a cube with small jump we have replaced $\bar u$ with a function $v^k \in \mathcal P(Q^{k})$ whose Mumford-Shah energy in $Q^k$ is controlled by the energy of $\bar u$ in $Q^{k}\cap P$. 
	
	If now $u^k$ has a large jump in $B^k_{r,r+s}$; \textit{i.e.,} if \eqref{GC1} is not satisfied, then we extend ${u}^k$ to $Q^k$ by simply setting
		\begin{equation}\label{vk-1}
			v^k:=  
			\begin{cases} 
			\displaystyle
			{u}^k 
			& \mbox{in } Q^{k}_{1/2,1}, \\
			0
			& \mbox{in } Q^k_{1/2}. 
			\end{cases}
		\end{equation}	
	Clearly $v^k \in \mathcal P(Q^{k})$, and
		\begin{align*}
			\mathcal{H}^{n-1}(S_{v^k} \cap Q^{k})
			&\leqslant  2n (1/2)^{n-1}+ \mathcal{H}^{n-1}(S_{\bar u} \cap Q^{k} \cap P)\\
			&< \frac{ 2n (1/2)^{n-1}}{\gamma s^{n-1}}\mathcal{H}^{n-1}(S_{{u}^k} \cap B^k_{r,r+s})
			+  \mathcal{H}^{n-1}(S_{\bar u} \cap Q^{k} \cap P)\
			\\
			&\leqslant 
			C  \mathcal{H}^{n-1}(S_{\bar u} \cap Q^{k} \cap P),
		\end{align*}
where $\omega_{n-1}$ is the surface of the unit sphere $\Sph^{n-1}$.
	Thus finally, for every $k \in \mathcal I^1$ we have replaced $\bar u$ with a function $v^k \in \mathcal P(Q^{k}) \subset SBV^{\textrm{pc}}(Q^{k})$ whose energy in $Q^k$ is controlled by the energy of $\bar u$ in $Q^{k}\cap P$; moreover, $\|\bar u\|_{L^\infty(Q^k\cap P)}=\|v^k\|_{L^\infty(Q^k)}$.

\medskip

	\noindent \textit{Step 2: Modification of $\bar u$ in the ``boundary" cubes $Q^k$, $k\in \mathcal I^2$.}
	In this step we consider only the cubes $Q^k$ such that $Q^k\cap \partial( tQ^\nu) \neq \emptyset$. In order to preserve the boundary condition we need to distinguish between two cases. If $Q^k\cap \partial( tQ^\nu \cap \{x\cdot \nu>0\}) \neq \emptyset$, in $Q^k$ we set 
\begin{equation}\label{vk-2}
v^k:= \begin{cases} 
			\displaystyle
			\bar{u}
			& \mbox{in } Q^{k}_{1/2,1}, \\
			1
			& \mbox{in } Q^{k}_{1/2},
			\end{cases}
			\end{equation}	
	while in those cubes $Q^k$ such that
	$
	Q^k\cap \partial( tQ^\nu \cap \{x\cdot \nu<0\}) \neq \emptyset,
	$
we set
\begin{equation}\label{vk-3}
v^k:= \begin{cases} 
			\displaystyle
			\bar{u}
			& \mbox{in } Q^{k}_{1/2,1}, \\
			0
			& \mbox{in } Q^{k}_{1/2}.
			\end{cases}
			\end{equation}			
The additional energy contribution of the boundary cubes is proportional to the perimeter of $t Q^\nu$\ie of order $Ct^{n-1}$ for some $C>0$ independent of $t$. 
\medskip			

	\noindent \textit{Step 3: Adding up all the cubes.} We now denote with $v\in\mathcal{P}(t Q^\nu)$ the function defined as $v=v^k$ in $Q^k$ for every $k$, where $v^k$ is as in \eqref{vk-0} or \eqref{vk-1} if $k\in \mathcal I^1$, as in \eqref{vk-2} if $k\in \mathcal I^2$ and $Q^k$ intersects $\partial( tQ^\nu \cap \{x\cdot \nu>0\})$, or \eqref{vk-3} if $k\in \mathcal I^2$ and $Q^k$ intersects $\partial( tQ^\nu \cap \{x\cdot \nu<0\})$.

	\medskip
	
By construction the function $v$ satisfies the following properties:	

	\begin{enumerate}
	        \item[i.] $ v\in \mathcal P(tQ^{\nu})$,
		
		\smallskip
		
		\item[ii.] $v =u_{0,1}^\nu$ in a neighbourhood of $\partial (tQ^{\nu})$, 
		
		\smallskip
		
		\item[iii.] $\mathcal{H}^{n-1}(S_{v} \cap (tQ^{\nu} \cap P)) \leqslant  \mathcal{H}^{n-1}(S_{\bar u} \cap (tQ^{\nu} \cap P))$,
		
		\smallskip
		
		\item[iv.] $\mathcal{H}^{n-1}(S_{v} \cap tQ^{\nu}) \leqslant C (\mathcal{H}^{n-1}(S_{\bar u} \cap (tQ^{\nu} \cap P)) + t^{n-1})$,
		
		\smallskip
		
		\item[iv.] $\|v\|_{L^{\infty}(tQ^{\nu})} = \|\bar u\|_{L^{\infty}(tQ^{\nu}\cap P)}=1$.
	\end{enumerate}

\smallskip
	
We finally set $w:=z v$ so that 	$w =u_{0,z}^\nu$ in a neighbourhood of $\partial (tQ^{\nu})$ and therefore it can be used as a competitor in the minimisation problem defining $g^\delta(z,\nu)$.
	
\medskip

	We are now able to compare $g^{\delta}(z,\nu)$ and $\hat{g}(\nu)$. By definition of $w$ have			
	\begin{align*}
			&\mathcal{H}^{n-1}(S_{w} \cap tQ^{\nu} \cap P) 
			+ \delta\int_{S_{w} \cap tQ^{\nu}}(1+|[w]|) \ d\mathcal{H}^{n-1} \\
			&\leqslant \mathcal{H}^{n-1}(S_{\bar u} \cap (tQ^{\nu} \cap P))
			+ \delta(1+|z|) \mathcal{H}^{n-1}(S_{v} \cap tQ^{\nu}) \\
			&\leqslant \mathcal{H}^{n-1}(S_{\bar u} \cap (tQ^{\nu} \cap P))
			+ \delta \,C(1+|z|)(\mathcal{H}^{n-1}(S_{\bar u} \cap (tQ^{\nu} \cap P)) + t^{n-1}),
	\end{align*}
	which is exactly the claim \eqref{claim-d-z}. Then in view of \eqref{g-delta-t} by dividing the above expression by $t^{n-1}$ and letting $t \rightarrow +\infty$, we get
		\begin{equation*}
			g^{\delta}(z,\nu) 
			\leqslant \lim\limits_{t \rightarrow +\infty}\dfrac{1}{t^{n-1}}
			\Big(\mathcal{H}^{n-1}(S_{\bar u} \cap (tQ^{\nu} \cap P))(1 + \delta\,C(1+|z|)\Big)
			+ \delta\,C(1+|z|).
		\end{equation*}
	By virtue of \eqref{quasi-min-gh} the previous estimate yields 
		\begin{equation*}
			g^{\delta}(z,\nu)
			\leqslant \hat{g}(\nu) \left(1 + \delta\, C(1+|z|)\right) +  \delta\, C(1+|z|),
		\end{equation*}
	and hence the claim.
	\end{proof}
	

\section{Identification of the homogenized  volume integrand}\label{sect:f}
	In this section we identify the limit volume integrand $f^\ell_{\rm{hom}}$. We start with a preliminary result.

\begin{figure}
\includegraphics{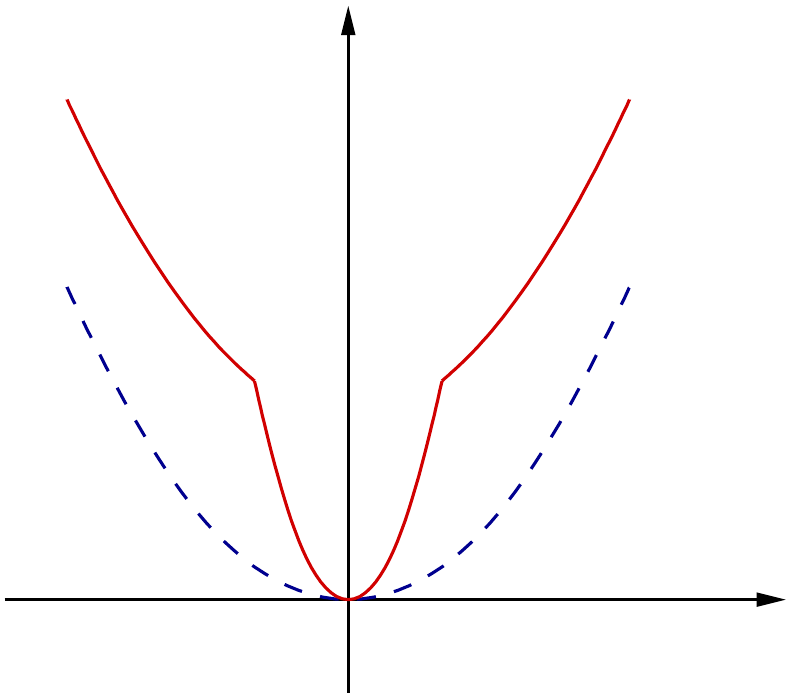}
\caption{Bounds on $f^\ell_{\rm hom}$}
\begin{picture}(0,0) 
\put(-10,50){{$0$}}
\put(90,50){{$\xi$}}
\put(55,110){$\hat f(\xi)$}
\put(60,180){$\min\{|\xi|^2,\hat f(\xi)+C\ell\}$}
\end{picture}
\end{figure}

\begin{lemma}
Let $f_{\rm{hom}}^\ell$ be as in Theorem \ref{t:int-rep} and let $\hat f$ be as in \eqref{F}. For every $\xi\in\mathbb{R}^{n}$ we have
		\begin{equation}
			\hat{f}(\xi) 
			\leqslant f^\ell_{\rm{hom}}(\xi) 
			\leqslant \min \{|\xi|^{2}, \hat{f}(\xi) 
			+ C \ell\},
			\label{ILV1}
		\end{equation}  
	where $C>0$ is a constant depending only on $P$ and on $\Omega$.
	\label{PL1}
\end{lemma}

\begin{proof}
	Let $\xi \in \mathbb{R}^{n}\setminus \{0\}$. By the classical homogenization result for the Dirichlet functional on a perforated domain \cite[Theorem 5.1]{Braides2}, there exists a sequence $(v_{\varepsilon}) \subset H^{1}(\Omega)$ which converges strongly in $L^{2}(\Omega)$, as $\e \to 0$, to the linear function $u_{\xi}(x):=\xi\cdot x$, such that $\|Dv_{\varepsilon}\|_{L^{2}(\Omega)} \leqslant C$, and 
		\begin{equation}
			\limsup_{\varepsilon \rightarrow 0}
			\int_{\Omega \cap \varepsilon P} |D v_{\varepsilon}|^2 dx 
			= \mathcal{L}^{n}(\Omega)\hat{f}(\xi),
			\label{sobolev-homogenization}
		\end{equation}
	where $\hat{f}$ is the quadratic form defined in \eqref{F}.	
	We define the new sequence $u_{\varepsilon}: \Omega \to \R$ as
		\begin{equation}
			u_{\varepsilon}:=  
			\begin{cases} 
			\displaystyle
			\smallskip
			v_{\varepsilon}
			& \mbox{in } Q^{k,\varepsilon}_{\frac{1}{2},1} \cap \Omega \\
			m_{\varepsilon}^{k}
			& \mbox{in } Q^{k,\varepsilon}_{\frac{1}{2}} \cap \Omega
			\end{cases}\qquad k\in \Z^n
			\label{recovering-sequence}
		\end{equation}	
	where for $0<s<r$ we set $Q^{k,\varepsilon}_{s,r}:=\varepsilon(Q_{s,r} + k)$ for $k\in\mathbb{Z}^{n}$, and 
		\begin{equation*}
			m^{k}_{\varepsilon} := \avint_{Q^{k,\varepsilon}_{\frac{1}{2}} \cap \Omega} v_{\varepsilon}(x) \ dx.
		\end{equation*}
	By \eqref{recovering-sequence} we immediately get
		\begin{equation*}
			\|v_{\varepsilon} - u_{\varepsilon}\|_{L^{2}(\Omega)}^{2} 
			= \sum_{k \in \mathbb{Z}^{n}} 
			\int_{Q^{k,\varepsilon}_{\frac{1}{2}} \cap \Omega} |v_{\varepsilon}(x) - m_{\varepsilon}^{k}|^{2} \ dx,
			\label{lemma-sequence}
		\end{equation*}
	moreover by the Poincar\'{e}-Wirtinger inequality, for every $k \in \Z^n$ we have
		\begin{equation*}\label{lemma-sequence2}
			\int_{Q^{k,\varepsilon}_{\frac{1}{2}} \cap \Omega} 
			|v_{\varepsilon}(x) - m_{\varepsilon}^{k}|^{2} \ dx 
			\leqslant C\e^{2}
			\int_{Q^{k,\varepsilon}_{\frac{1}{2}}\cap \Omega} |D v_{\varepsilon}(x) |^{2} \ dx,
		\end{equation*}
	with a constant $C>0$ independent of $\varepsilon$ and of $k$. 
	From the bound $\|D v_{\varepsilon}\|_{L^{2}(\Omega)} \leqslant C$ we eventually deduce 
		\begin{equation*}
			\|v_{\varepsilon} -u_{\varepsilon}\|_{L^{2}(\Omega)}		
		\leqslant C \varepsilon,
			\label{control-lemma}
		\end{equation*}
	therefore, since $v_{\varepsilon} \rightarrow u_{\xi}$ strongly in $L^{2}(\Omega)$, also $(u_{\varepsilon})$ converges to $u_{\xi}$ strongly in $L^{2}(\Omega)$. 
	We now estimate 
		\begin{align}
			F_{\varepsilon}(u_{\varepsilon}) 
			&= \int_{\Omega}|\nabla u_{\varepsilon}|^{2} dx 
			+ \mathcal{H}^{n-1}(S_{u_{\varepsilon}} \cap (\Omega \cap \varepsilon P))
			+ \beta_{\varepsilon} \mathcal{H}^{n-1}(S_{u_{\varepsilon}} \cap (\Omega 
			\setminus \varepsilon P)) \nonumber\\
			&\leqslant \int_{\Omega \cap \varepsilon P}|D v_{\varepsilon}|^{2} dx 
			+ \beta_{\varepsilon} \mathcal{H}^{n-1}(S_{u_{\varepsilon}} \cap \partial(\Omega 
			\setminus \varepsilon P)).	
			\label{control-lemma2}			
		\end{align} 
	For the second term in \eqref{control-lemma2} we have
		\begin{equation*}
			\beta_{\varepsilon} 
			\mathcal{H}^{n-1}(S_{u_{\varepsilon}} \cap \partial(\Omega \setminus \varepsilon P)) 
			\leqslant \beta_{\varepsilon}N(\varepsilon) 
			\mathcal{H}^{n-1}(\partial Q_{\e/2})
			+ \beta_{\varepsilon} \mathcal{H}^{n-1}(\partial \Omega),
			\end{equation*}
	where $N(\varepsilon)$ is the cardinality of the set $\{k\in\Z^n: \overline{Q^{k,\varepsilon}}\cap \Omega\neq \emptyset\}$. Since $\Omega$ is a bounded set with Lipschitz boundary, we have that $N(\e)\leqslant c/\e^n$ and $ \mathcal{H}^{n-1}(\partial \Omega)<c$, which gives 
		\begin{equation*}
			F_{\varepsilon}(u_{\varepsilon}) 
			\leqslant \int_{\Omega \cap \varepsilon P}|D v_{\varepsilon}|^{2} dx  
			+ C \dfrac{\beta_{\varepsilon}}{\varepsilon}.
		\end{equation*}
	Now, as in particular $u_{\varepsilon} \rightarrow u_{\xi}$ in the sense of Definition \ref{def:conv},  
	by Theorem \ref{Gamma-compactness} and Theorem \ref{t:int-rep}, we have
		\begin{equation*}
			\mathcal{L}^{n}(\Omega) f^\ell_{\rm{hom}}(\xi)
			= F^{\ell}_{\rm{hom}}(u_{\xi})
			\leqslant \liminf_{\varepsilon \rightarrow 0}
			F_{\varepsilon}(u_{\varepsilon})
			\leqslant \limsup_{\varepsilon \rightarrow 0} 
			\int_{\Omega \cap \varepsilon P}|D v_{\varepsilon}|^{2} dx
			+ C \ell,
		\end{equation*}
	which, together with \eqref{sobolev-homogenization}, gives 
		\begin{equation*}
			\mathcal{L}^{n}(\Omega) f^\ell_{\rm{hom}}(\xi)
			\leqslant \limsup_{\varepsilon \rightarrow 0} 
			\int_{\Omega \cap \varepsilon P}|D v_{\varepsilon}|^{2} dx
			+ C \ell
			\leqslant \mathcal{L}^{n}(\Omega) \hat{f}(\xi)
			+ C \ell.
		\end{equation*}
	Dividing by $\mathcal{L}^{n}(\Omega)$ and using \eqref{c:bd-on-fhom} concludes the proof.
	\end{proof}

\subsection{Subcritical case: $\ell = 0$.}
	\begin{theo}[Identification of the homogenized  volume integrand for $\ell=0$]\label{t:f-hom-f-hat}
	Let $f_{\rm{hom}}^0$ be as in Theorem \ref{t:int-rep} and corresponding to the choice $\ell=0$, and let $\hat{f}$ be as in \eqref{F}.
	Then for every $\xi \in \mathbb{R}^{n}$ we have $f_{\rm{hom}}^0(\xi) = \hat{f}(\xi)$.
	\end{theo}

\begin{proof}
	The thesis readily follows from \eqref{ILV1} by using the assumption that $\ell=0$, and the upper bound in \eqref{B1}.
\end{proof}

\begin{rem}
Note that, although in Theorem \ref{Gamma-compactness} the $\Gamma$-convergence of $F_\e$ has been established only up to subsequences (and the $\Gamma$-limit might be in principle different along different subsequences), in the subcritical case the situation is different. Indeed, thanks to Theorem \ref{t:g-hom-g-hat} and Theorem \ref{t:f-hom-f-hat}, we deduce that the $\Gamma$-limit is the same for every subsequence, and is given by the functional $\widehat{F}$ defined in \eqref{def:Fhat}.
\end{rem}

\medskip

\subsection{Supercritical case: $\ell = +\infty$.}

\begin{theo}[Identification of the homogenized  volume integrand for $\ell=+\infty$]\label{id:f_infty} 
	Let $f_{\rm{hom}}^\infty$ be as in Theorem \ref{t:int-rep} and corresponding to the choice $\ell=+\infty$.
	Then, for every $\xi \in \mathbb{R}^{n}$, we have that $f_{\rm{hom}}^{\infty}(\xi) = |\xi|^2$. 
	
	\end{theo}

\begin{rem}
Theorem \ref{t:g-hom-g-hat} and Theorem \ref{id:f_infty} imply that, for $\ell=+\infty$, the $\Gamma$-limit of $(F_{\varepsilon})$ is 
\begin{equation*}
			{F}^\infty_{\rm{hom}}(u) := 
			\begin{cases} 
			\displaystyle
			\int_{\Omega}|\nabla u|^2 dx 
			+ \int_{S_{u}} \hat{g}(\nu_u) d\mathcal{H}^{n-1}
			& \mbox{if} \ u \in GSBV^{2}(\Omega), \\ 
			+\infty
			& \mbox{otherwise in } L^{1}(\Omega).
			\end{cases}
		\end{equation*}
In particular, the whole sequence $(F_{\varepsilon})$ $\Gamma$-converges to $F_{\rm{hom}}^\infty$.	
\end{rem}

Before proving Theorem \ref{id:f_infty} above we need to recall the Elimination Property proved in \cite[Lemma 0.7]{DMMS92} (see also \cite[Theorem 3.6]{DeGiorgi}). For the definition of local minimiser of the Mumford-Shah functional we refer to \cite[Definition 6.6]{AFP}.

\begin{theo}[Elimination property]
Let $\Omega \subset\mathbb{R}^{n}$ be open. There exists a strictly positive dimensional constant 
$\theta = \theta(n)$ independent of $\Omega$ such that, if $u\in SBV^{2}(\Omega)$ is a local minimiser of the Mumford-Shah functional and $B_{\rho}(x_{0}) \subset \Omega$ is any ball with centre $x_{0}$ and with
	\begin{equation*}
		\mathcal{H}^{n-1}(S_{u} \cap B_{\rho}(x_{0})) < \theta \rho^{n-1},
	\end{equation*} 
then $S_{u} \cap B_{\frac{\rho}{2}}(x_{0})=\emptyset$.	
	\label{elimination}
\end{theo}

We now introduce some auxiliary functionals which will be used in the proof of Theorem \ref{id:f_infty}.

\medskip

Let $\frac{1}{2}<r<1$ and let $\varphi\in H^{1/2}(\partial Q_r)$. For $h\in \N$ and $t>1$ we define the functionals $I_{\varphi}, I^{h,t}_\varphi: L^{1}(Q_r) \longrightarrow [0,+\infty]$ as follows:
\begin{equation}
		{I}_{\varphi}(u) := 
		\begin{cases} 
		\displaystyle
		\int_{Q_{r}}|D u|^{2} \dx
		& \mbox{if} \ u \in H^{1}(Q_r), \ u = \varphi \ \text{ on } \partial Q_r, \\ 
		+\infty
		& \mbox{otherwise in } L^{1}(Q_r),
		\end{cases}
		\label{DI1}
	\end{equation}
and
\begin{equation}
		I^{h,t}_{\varphi}(u) := 
		\begin{cases} 
		\displaystyle
		\int_{Q_{r}}|\nabla u|^{2} dx
		+ t \mathcal{H}^{n-1}(S_{u} \cap Q_{\frac{1}{2},r})
		+ \mathcal{H}^{n-1}(S_{u} \cap Q_{\frac{1}{2}}) \\
		\qquad \qquad \text{if} \ u \in SBV^{2}(Q_{r}),  
		 \ \mathcal{H}^{n-1}(S_{u}) \leqslant \dfrac{1}{h}, \ w = \varphi \ \text{ on } \partial Q_{r}, \\ 
		+\infty 
		\qquad \ \text{otherwise in} \ L^{1}(Q_{r}).
		\end{cases}
		\label{DI2}
	\end{equation}

\noindent The next result is a straightforward adaption of \cite[Lemma 4.3]{SLD1}.

\begin{lemma}\label{GammaMin}
Let $\frac{1}{2}<r<1$;  
let $\varphi \in H^{1/2}(\partial Q_r)$, and let $(\varphi_h) \subset H^{1/2}(\partial Q_{r})$ be a sequence with
	\begin{equation*}
		\varphi_{h} \rightarrow \varphi \text{ in } H^{1/2}(\partial Q_r), \ \textrm{ as } h\to +\infty.
	\end{equation*}
Then the functionals ${I}_{\varphi_{h}}$ and ${I}_{\varphi_{h}}^{h,t}$ defined, respectively, as in \eqref{DI1} and \eqref{DI2}, with $\varphi$ replaced by $\varphi_h$, $\Gamma$-converge with respect to the strong $L^1(\Omega)$-topology, as $h\to +\infty$, to the Dirichlet functional 
$I_{\varphi} : L^{1}(Q_{r}) \rightarrow [0,+\infty]$ defined in  \eqref{DI1}.
\end{lemma}

We now state and prove a technical result which is the heart of the proof of Theorem \ref{id:f_infty}.

\begin{theo}[Lower bound for $F_\varepsilon$]\label{ILVII}
Let $\xi\neq0$ and let $(u_{\varepsilon})\subset L^1(\Om)$ be a sequence such that $\sup_\e F_{\varepsilon}(u_{\varepsilon})<+\infty$ and $u_\e \to u_{\xi}$ in the sense of Definition \ref{def:conv}. Then
	\begin{equation}\label{FL4}
		\liminf_{\varepsilon \rightarrow 0}F_{\varepsilon}(u_{\varepsilon}) \geqslant \mathcal{L}^{n}(\Omega)|\xi|^{2}.
	\end{equation} 
\end{theo}

\begin{proof} 
Assume that $F_\e(u_\e)\leqslant c$. The proof strategy consists of replacing the sequence $u_{\varepsilon}$ with an \textit{improved} sequence $w_{\varepsilon}$ (in a sense that will be clarified below) which converges to $u_{\xi}$ strongly in $L^{1}(\Omega)$, and whose energy is asymptotically smaller than $F_\e(u_\e)$. 

Since the energy $F_\e$ decreases by truncations, we can truncate the sequence $(u_\e)$ at level $\|u_\xi\|_{L^\infty(\Om)}$ and preserve both the uniform bound on $F_\e(u_\e)$ and the convergence of $(u_\e)$ to $u_\xi$. Hence in what follows we assume that $\|u_\e\|_{L^\infty(\Om)}\leq \|u_\xi\|_{L^\infty(\Om)}$. 

\smallskip
 
As an initial step we rewrite $\Omega$ as
	\begin{equation*}
		\Omega 
		= \Bigg(\bigcup\limits_{k\in \mathcal{I}^1_\e}Q^{k,\varepsilon}\Bigg) 
		\cup \Bigg(\bigcup\limits_{k\in \mathcal{I}_\e^2}Q^{k,\varepsilon}\cap \Om\Bigg),
	\end{equation*} 
where $Q^{k,\varepsilon}:= \varepsilon(Q + k)$, $\mathcal{I}^1_\e:=\{z\in \Z^n: Q^{k,\varepsilon}\subset \Om\}$, and $\mathcal{I}^2_\e:=\{z\in \Z^n: Q^{k,\varepsilon}\cap \partial \Om\neq \emptyset\}$.
Clearly  
	\begin{equation*}
		F_{\varepsilon}(u_{\varepsilon})
		= \sum_{k\in\mathcal{I}^1_\e} F_{\varepsilon}(u_{\varepsilon},Q^{k,\varepsilon})
		+ \sum_{k\in\mathcal{I}^2_\e}
		F_{\varepsilon}(u_{\varepsilon},Q^{k,\varepsilon}\cap\Omega).
	\end{equation*}
	
	\medskip
	
\noindent\textit{Step 1: Classification of the interior cubes.}
We estimate the energy in $Q^{k,\varepsilon}$, for $k \in \mathcal{I}^1_\e$, namely
	\begin{equation*}
		F_{\varepsilon}(u_{\varepsilon}, Q^{k,\varepsilon})
		= \int_{Q^{k,\varepsilon}}|\nabla u_{\varepsilon}|^{2} \ dx
		+ \mathcal{H}^{n-1}(S_{u_{\varepsilon}} \cap Q^{k,\varepsilon}_{\frac{1}{2},1}) 
		+ \beta_{\varepsilon}
		\mathcal{H}^{n-1}(S_{u_{\varepsilon}} \cap Q^{k,\varepsilon}_{\frac{1}{2}}).
	\end{equation*}
For $y\in Q^{k}$ set $v_{\varepsilon}(y):=(\sqrt{\beta_{\varepsilon}\varepsilon})^{-1} u_{\varepsilon}(\varepsilon y)$; by changing variables we have
	\begin{align*}
		F_{\varepsilon}(u_{\varepsilon}, Q^{k,\varepsilon})
		&= \beta_{\varepsilon}\varepsilon^{n-1}\bigg(\int_{Q^{k}}|\nabla v_{\varepsilon}|^{2} \ dy
		+ \dfrac{1}{\beta_{\varepsilon}}
		\mathcal{H}^{n-1}(S_{v_{\varepsilon}} \cap Q^{k}_{\frac{1}{2},1}) 
		+ \mathcal{H}^{n-1}(S_{v_{\varepsilon}} \cap Q^{k}_{\frac{1}{2}})\bigg) \nonumber\\
		&=: \beta_{\varepsilon}\varepsilon^{n-1}
		\mathcal{F}_{\varepsilon}(v_{\varepsilon},Q^{k}).
	\end{align*}
Let $\vartheta>0$ be a fixed constant, and let $\varepsilon>0$ be fixed. We call $Q^{k}$ a \textit{good} cube if 
	\begin{equation}
		\mathcal{F}_{\varepsilon}(v_{\varepsilon},Q^{k}) \leqslant C \quad \textrm{for some } C>0 \quad \textrm{and }  \quad 
		 \mathcal{H}^{n-1}(S_{v_{\varepsilon}} \cap Q^{k}) \leqslant \vartheta,
		\label{GC}
	\end{equation}
	namely if both the energy of $v_\e$ and the total measure of the jump of $v_\e$ are bounded in the cube. Otherwise, we say that $Q^{k}$ is a \textit{bad} cube.
Let $\mathcal{I}_\e^{1,g}, \mathcal{I}^{1,b}_\e \subset \mathcal{I}^1_\e$ denote the set of internal good and bad cubes, respectively; we denote with $N_{\e}^{g}$ and $N_{\e}^{b}$ their cardinalities.

We can easily estimate the number of bad cubes, by using the fact that either one of the conditions \eqref{GC} is not satisfied. Namely, if the first condition in \eqref{GC} is not satisfied, then 
$$
c \geqslant F_{\varepsilon}(u_{\varepsilon}) \geqslant\beta_{\varepsilon}\varepsilon^{n-1}\sum_{k\in\mathcal{I}_\e^{1,b}}  \mathcal{F}_{\varepsilon}(v_{\varepsilon},Q^{k}) \geqslant C \beta_{\varepsilon}\varepsilon^{n-1} N_{\e}^{b}.
$$			
Similarly, if the second condition in \eqref{GC} is not satisfied, then 
	\begin{equation*}
	c \geqslant \beta_{\varepsilon}\varepsilon^{n-1}N_{\e}^{b}\vartheta.
	\end{equation*}
Hence, we have the bound 
	\begin{equation}\label{num_bc}
		N_{\e}^{b}
		\leqslant  \dfrac{C(\vartheta)}{\beta_\e\varepsilon^{n-1}}.
	\end{equation}
\medskip

\noindent\textit{Step 2: Energy estimate on the good cubes.} This is the most delicate part of the proof, and is split into a number of sub-steps.
\smallskip

\noindent
\textit{Step 2.1: Elimination property.}
Let $Q^{k}$, for $k\in \mathcal{I}_\e^{1,g}$, be an arbitrary good cube in the sense of \eqref{GC}. 
We omit now the superscript $k$ for the sake of notation. 
Let $\frac{1}{2}<\varsigma<1$; for fixed $\varepsilon>0$ consider the following (local) minimisation problem:
	\begin{align*}
		\textrm{(LMS)}_\e \
		 \text{loc min}\bigg\{
		\mathcal{F}_{\varepsilon}(v, Q_{\varsigma}):  v \in SBV^{2}(Q_{\varsigma}),
		\mathcal{F}_{\varepsilon}(v, Q_{\varsigma})
		\leqslant C \ \textrm{for some } C>0,  \mathcal{H}^{n-1}(S_{v} \cap Q_{\varsigma}) 
		\leqslant
		\vartheta
		\bigg\}.
	\end{align*}
Let $\mathcal{M}^\e_{\vartheta}$ denote the class of solutions of $\textrm{(LMS)}_\e$, and let $\hat{v}_{\varepsilon}\in \mathcal{M}^\e_{\vartheta}$. We recall that, following \cite[Definition 6.6]{AFP}, for every open set $A\subset \subset  Q_{\varsigma}$ we have that $\mathcal{F}_{\varepsilon}(\hat v_\e, A) \leq \mathcal{F}_{\varepsilon}(v, A)$, whenever $\{\hat v_\e \neq v\}\subset \subset A$.

With no loss of generality we can assume that $\hat{v}_{\varepsilon} $ is bounded in $L^\infty(Q_\varsigma)$ (with a possibly $\e$-dependent constant) as $v_\e$, since the energy $\mathcal{F}_\e$ decreases by truncations. We observe that $\hat{v}_{\varepsilon}$ is also a local minimiser of $\mathcal{F}_{\varepsilon}(\cdot,Q_{\frac{1}{2},\varsigma})$.

By setting $\hat{w}_{\varepsilon}(y) := \sqrt{\beta_{\varepsilon}}\hat{v}_{\varepsilon}(y)$, we have
	\begin{equation*}
		\mathcal{F}_\e(\hat{v}_{\varepsilon},Q_{\frac{1}{2},\varsigma})
		= \dfrac{1}{\beta_{\varepsilon}}\bigg(
		\int_{Q_{\frac{1}{2},\varsigma}}|\nabla \hat{w}_{\varepsilon}|^{2} dy
		+ \mathcal{H}^{n-1}(S_{\hat{w}_{\varepsilon}} \cap Q_{\frac{1}{2},\varsigma})\bigg),
	\end{equation*}
and 
\begin{equation*}
		\mathcal{H}^{n-1}(S_{\hat{v}_{\varepsilon}} \cap Q_{\frac{1}{2},\varsigma})
		=\mathcal{H}^{n-1}(S_{\hat{w}_{\varepsilon}} \cap Q_{\frac{1}{2},\varsigma}) 
		\leqslant C \beta_{\varepsilon}.
	\end{equation*}
Hence the function $\hat{w}_{\varepsilon}$ is a local minimiser of the Mumford-Shah functional 
in $Q_{\frac{1}{2},\varsigma}$, and its jump set has small Hausdorff measure. We then apply Theorem \ref{elimination} to $\hat w_\e$;
namely, we fix $x_{0} \in Q_{\frac{1}{2},\varsigma}$ and consider a ball $B_{\rho(\varepsilon)}(x_{0})\subset\subset Q_{\frac{1}{2},\varsigma}$, where $\rho(\varepsilon)$ is such that
	\begin{equation}\label{elimination-radius}
		\mathcal{H}^{n-1}(S_{\hat{w}_{\varepsilon}} \cap B_{\rho(\varepsilon)}(x_{0}))
		\leqslant C\beta_{\varepsilon}
		\leqslant \theta\rho(\varepsilon)^{n-1}
	\end{equation}
is satisfied for $\varepsilon>0$ fixed, where $\theta$ is the elimination constant in Theorem \ref{elimination} (note that $\rho(\e)$ can be very small, of the order of $(\beta_{\varepsilon})^{\frac{1}{n-1}}$). Then Theorem \ref{elimination} guarantees that
	\begin{equation*}
		S_{\hat w_\e}\cap B_{\frac{\rho(\varepsilon)}{2}}(x_{0}) = \emptyset.
	\end{equation*}
This same argument can be repeated for every $x \in Q_{\frac{1}{2},\varsigma}$ with any radius $\rho(\varepsilon)>0$ satisfying \eqref{elimination-radius}. 
In this way we conclude that the jump of $\hat{w}_{\varepsilon}$ has to be contained in a neighbourhood of $\partial Q_{\frac{1}{2},\varsigma}$ of order 
$ (\beta_{\varepsilon})^{\frac{1}{n-1}}$. In particular, there exist $\frac{1}{2}<\rho_1<\rho_2<\varsigma$ such that $Q_{\rho_1,\rho_2}\subset\subset Q_{\frac{1}{2},\varsigma}$ and $S_{\hat w_\e}\cap \overline{Q}_{\rho_1,\rho_2} = \emptyset$.

From the definition of $\hat w_\e$ we deduce that, for any $\hat v_\e \in \mathcal{M}^\e_{\vartheta}$ (and for sufficiently small $\e$)
	\begin{equation*}
		\mathcal{H}^{n-1}(S_{\hat{v}_{\varepsilon}} \cap \overline{Q}_{\rho_1,\rho_2})
		= 0,
	\end{equation*}
hence $\hat{v}_{\varepsilon} \in H^{1}(Q_{\rho_1,\rho_2})$, and $S_{\hat{v}_{\varepsilon}} \subseteq Q_{\rho_1}\cup Q_{\rho_2,\varsigma}$.
Moreover, since the Mumford-Shah functional is invariant under translations, we can assume with no loss of generality that any local minimiser $\hat v_\e \in \mathcal{M}^\e_{\vartheta}$ satisfies
	\begin{equation}\label{normalisation}
		\int_{Q_{\rho_1,\rho_2}} \hat{v}_\e \dx = 0.
	\end{equation} 	 
\textit{Step 2.2: Comparison between $\hat v_\e$ and its harmonic extension.} For a given $\hat{v}_{\varepsilon} \in \mathcal{M}_{\vartheta}^{\varepsilon}$ we define the function 
$\tilde{v}_{\varepsilon} \in H^{1}(Q_{\rho_2})$ as the solution of the following Dirichlet problem for the Laplace equation
	\begin{equation*}
		\textrm{(Dir)}  \
		\begin{cases}
		\Delta w = 0  &\text{ in } Q_{\rho_1},  \\
		w = \hat{v}_{\varepsilon}   &\text{ in } \overline{Q}_{\rho_1,\rho_2}.
		\end{cases}
	\end{equation*}
Throughout this sub-step we simply write $\hat v$ and $\tilde v$ instead of $\hat v_\e$ and $\tilde v_\e$.

We now claim that for every $\eta>0$ there exists $\vartheta^\ast=\vartheta^\ast(\eta)>0$ such that for every $\hat{v} \in \mathcal{M}^\e_{\vartheta^\ast}$ and every corresponding $\tilde{v}$ as in (Dir) we have 
	\begin{equation}
		\int_{Q_{\rho_2}}|D \tilde{v}|^{2} \dx \leqslant (1+\eta) \int_{Q_{\rho_2}} |\nabla \hat{v}|^{2} \dx.
		\label{FI}
	\end{equation}
We note that the claim is true for constant $\tilde v$. Therefore we only need to prove  \eqref{FI} when $\tilde v$ is not constant. Arguing by contradiction we assume that there exists $\eta>0$ such that for every $h\in \N$ there are $\hat{v}_{h} \in \mathcal{M}^\e_{1/h}$ and $\tilde{v}_{h}$ defined as in (Dir) satisfying 
	\begin{equation*}
	\int_{Q_{\rho_2}}|D \tilde{v}_{h}|^{2} \dx > (1+\eta) \int_{Q_{\rho_2}} |\nabla \hat{v}_{h}|^{2} \dx.
	\end{equation*}
Using that $\hat{v}_{h} = \tilde{v}_{h}$ in $Q_{\rho_1,\rho_2}$, and that $\hat v^h \in H^{1}(Q_{\rho_1,\rho_2})$ by Step 2.1, the previous estimate gives
	\begin{equation}
	\int_{Q_{\rho_1}}|D \tilde{v}_{h}|^{2} \dx>
		(1+\eta) \int_{Q_{\rho_1}} |\nabla \hat{v}_{h}|^{2} \dx
		+ \eta \int_{Q_{\rho_1,\rho_2}} |D \hat{v}_{h}|^{2} \dx
	\label{contradiction}
	\end{equation}
From the normalisation condition in \eqref{normalisation} and the energy bound in $\textrm{(LMS)}_\e$ satisfied by $\hat v_\e$, we can apply the Poincar\'{e}-Wirtinger inequality to deduce that there exists a constant $C>0$ (independent of $h$) such that $\|\hat{v}_{h}\|_{H^{1}(Q_{\rho_1,\rho_2})} \leqslant C$. Therefore $\hat{v}_{h}$ converges weakly in $H^1(Q_{\rho_1,\rho_2})$, hence in particular  
	\begin{equation}
		\varphi_{h}:= (\hat{v}_{h})_{|\partial Q_{\rho_1}} \rightarrow \varphi
		\text{ strongly in } H^{1/2}(\partial Q_{\rho_1}), 
		\label{TraceCV}
	\end{equation}
for some $\varphi \in H^{1/2}(\partial Q_{\rho_1})$. 
Moreover, $(\tilde{v}_{h})_{|\partial Q_{\rho_1}}= \varphi_{h}$. Then, since 
$$
\|\tilde v_h\|_{H^1(Q_{\rho_1})} \leqslant C\|\varphi_h\|_{H^{1/2}(\partial Q_{\rho_1})},
$$
we immediately deduce that $\tilde v_h$ is uniformly bounded in $H^1(Q_{\rho_1})$.  

	We now apply Lemma \ref{GammaMin} with $r=\rho_1$, $t=\frac1{\beta_\e}$ and the functions $\varphi_h$ and $\varphi$ defined in \eqref{TraceCV}. By the fundamental theorem of $\Gamma$-convergence, the sequence $\tilde{v}_{h}$, which is a compact sequence of minimisers for the functionals $I_{\varphi_{h}}$, converges weakly in $H^1(Q_{\rho_1})$  to the unique minimiser of $I_{\varphi}$, which we denote with $\tilde{v}$. 
Furthermore we have convergence of the corresponding minimum values:
	\begin{align*}
		\lim_{h \rightarrow +\infty} \int_{Q_{\rho_1}}|D \tilde{v}_h|^{2} \dx
		&= \lim_{h \rightarrow +\infty} 
		I_{\varphi_{h}}(\tilde{v}_{h})
		= \lim_{h \rightarrow +\infty} 
		\inf_{v \in H^1(Q_{\rho_1})} {I}_{\varphi_{h}}(v)\\
		&= \min_{v \in H^{1}(Q_{\rho_1})} I_{\varphi}(v)
		= \int_{Q_{\rho_1}}|D \tilde{v}|^{2} \dx. 
	\end{align*}	
Similarly, the sequence $\hat{v}_{h} \in \mathcal{M}^\e_{1/h}$, which is a compact sequence of minimisers of the functionals $I_{\varphi_{h}}^{h,t}$, converges in $L^1(Q_{\rho_1 })$ to the unique minimiser of $I_{\varphi}$\ie to $\tilde v$. 
Furthermore we have convergence of the minimum values:
	\begin{equation*}
		\lim_{h \rightarrow +\infty} 
		I_{\varphi_{h}}^{h,t}(\hat{v}_{h})
		= \lim\limits_{h \rightarrow +\infty} 
		\inf_{v \in \mathcal{M}_{1/h}^{\varepsilon}} {I}^{h,t}_{\varphi_{h}}(v)
		= \min_{v \in H^{1}(Q_{\rho_1})} I_{\varphi}(v)
		= \int_{Q_{\rho_1}}|D \tilde{v}|^{2} \dx. 
	\end{equation*}	
On the other hand, we clearly have, by the definition of good cubes  \eqref{GC}, with $\vartheta=\frac1h$, that 
	\begin{equation*}
		\lim_{h \rightarrow +\infty} 
		{I}^{h,t}_{\varphi_{h}}(\hat{v}_{h})
		= \lim_{h \rightarrow +\infty} 
		\int_{Q_{\rho_1}}|\nabla \hat{v}_{h}|^{2} \dx,
	\end{equation*}
and hence 
$$
\lim_{h \rightarrow +\infty} \int_{Q_{\rho_1}}|\nabla \hat{v}_{h}|^{2} \dx = \int_{Q_{\rho_1}}|D \tilde{v}|^{2} \dx.
$$
By passing to the limit in \eqref{contradiction} we then have in particular that
\begin{equation*}
(1+\eta) \int_{Q_{\rho_1}}|D \tilde{v}|^{2} \dx
		\leq \int_{Q_{\rho_1}}|D \tilde{v}|^{2} \dx,
\end{equation*}	
which gives a contradiction since $D\tilde u\not \equiv 0$ and $\eta>0$, and thus proves \eqref{FI}. 

\medskip

In view of  \eqref{FI}, in what follows we will choose $\vartheta=\vartheta^\ast$ in the definition of good cubes \eqref{GC}.

\medskip

\noindent \textit{Step 2.3: Energy bound on the good cubes.} Let $\vartheta^\ast$ be as in Step 2.2; we consider the following minimisation problems
\begin{align*}
		\textrm{(MS)}_\e 
		\ \min\bigg\{
		& \mathcal{F}_\e(v, Q_{\rho_2}):  \ v \in SBV^{2}(Q_{\rho_2}), \ S_{v} \subset Q_{\rho_1}, \mathcal{H}^{n-1}(S_{v} \cap Q_{\rho_2}) \leqslant
		\vartheta^\ast,
		\\  
		& \hspace{2cm}
		\mathcal{F}_{\varepsilon}(v, Q_{\rho_2})
		\leqslant C \ \textrm{ for some } C>0, \ v = v_{\varepsilon}\ \text{ on } 
		\partial Q_{\rho_2}
		\bigg\},
	\end{align*}
	where $\frac{1}{2}<\rho_1<\rho_2<1$.
For a minimiser $\hat{v}_{\varepsilon}$ of $\textrm{(MS)}_\e $, let $\tilde{v}_{\varepsilon}$ be the corresponding function as in (Dir). Since $\hat{v}_{\varepsilon}$ is also a local minimiser of the same functional, from \eqref{FI} 
we have
	\begin{equation*}
		\int_{Q_{\rho_2}}|D \tilde{v}_\e|^{2} \dx \leqslant (1+\eta) \int_{Q_{\rho_2}} |\nabla \hat{v}_\e|^{2} \dx.
	\end{equation*}
Then for the sequence $(v_\e)$ defined at the beginning of Step 1 we have    
	\begin{align*}
		\mathcal F(v_\e,Q_{\rho_2}) &= \int_{Q_{\rho_2}}|\nabla v_{\varepsilon}|^{2} \ dy 
		+ \dfrac{1}{\beta_{\varepsilon}} 
		\mathcal{H}^{n-1}\left(S_{v_{\varepsilon}} \cap Q_{\frac{1}{2},\rho_2} \right) 
		+ \mathcal{H}^{n-1}\left(S_{v_{\varepsilon}} \cap Q_{\frac{1}{2}}\right) \\
		&\geqslant \int_{Q_{\rho_2}}|\nabla \hat{v}_{\varepsilon}|^{2} \ dy 
		+ \dfrac{1}{\beta_{\varepsilon}} 
		\mathcal{H}^{n-1}\left(S_{\hat{v}_{\varepsilon}} \cap Q_{\frac{1}{2},\rho_2} \right) 
		+ \mathcal{H}^{n-1}\left(S_{\hat{v}_{\varepsilon}} \cap Q_{\frac{1}{2}}\right) \\
		&\geqslant \left(1 - \dfrac{\eta}{1+\eta}\right)\int_{Q_{\rho_2}}|D \tilde{v}_{\varepsilon}|^{2} \ dy.  		 
	\end{align*}
By the change of variables 
$\tilde{u}_{\varepsilon}(\varepsilon y) := \sqrt{\beta_{\varepsilon}\varepsilon}\tilde{v}_{\varepsilon}(y)$, we get
	\begin{align*}
		F(u_\e,Q^\e_{\rho_2}) = & \int_{Q_{\rho_2}^{\varepsilon}}|\nabla u_{\varepsilon}|^{2} \dx 
		+ \mathcal{H}^{n-1}\left(S_{u_{\varepsilon}} \cap Q_{\frac{1}{2},\rho_2}^{\varepsilon} \right) 
		+ \beta_{\varepsilon}
		\mathcal{H}^{n-1}\left(S_{u_{\varepsilon}} \cap Q_{\frac{1}{2}}^{\varepsilon} \right) \\ 
		&\geqslant \left(1 - \dfrac{\eta}{1+\eta}\right)\int_{Q_{\rho_2}^{\varepsilon}}
		|D \tilde{u}_{\varepsilon}|^{2} \dx.
	\end{align*}
Hence, for every $k\in \mathcal{I}_\e^{1,g}$, 
	\begin{equation}
		F_{\varepsilon}(u_{\varepsilon},Q_{\rho_2}^{k,\varepsilon})
		\geqslant \left(1 - \dfrac{\eta}{1+\eta}\right)\int_{Q_{\rho_2}^{k,\varepsilon}}
		|D \tilde{u}_{k,\varepsilon}|^{2} \ dx,
		\label{GCubes}
	\end{equation}
where the subscript $k$ has now been added to highlight the dependence of the construction of $\tilde{u}_{k,\varepsilon}$ on the cube $Q^{k,\varepsilon}$.

\medskip

\noindent \textit{Step 3: Energy estimate on the bad cubes and on the boundary cubes.} Let now $k\in \mathcal{I}_\e^{1,b}\cup \mathcal{I}_\e^{2}$. 
We bound the energy of $u_\e$ on $Q^{k,\varepsilon}\cap \Om$ as
	\begin{equation}
		F_{\varepsilon}(u^{\varepsilon},Q^{k,\varepsilon}\cap \Om)
		\geqslant \widehat{F}_\e(\hat{u}_{k,\varepsilon},Q^{k,\varepsilon}\cap \Om),
		\label{BCubes}
	\end{equation}
where $\widehat{F}_\e$ is defined as in \eqref{widehatFe} and $\hat{u}_{k,\varepsilon} : Q^{k,\varepsilon}\cap \Om \rightarrow \mathbb{R}$ is defined as
	\begin{equation}\label{uhatk}
		\hat{u}_{k,\varepsilon}:= T_\e\left((u_{\varepsilon})_{|(Q^{k,\varepsilon} \cap \varepsilon P)\cap\Om}\right),
	\end{equation}
where $T_\e$ denotes the extension operator provided by \cite[Theorem 1.1]{CS1}. 
 	
	\medskip

\noindent \textit{Step 4: Construction of an improved sequence converging to $u_\xi$.} We define the sequence $w_{\varepsilon}:\Omega\to\R$ as
	\begin{equation*}
		w_{\varepsilon} := 
		\begin{cases} 
		\smallskip
		\displaystyle
		\tilde{u}_{k,\varepsilon}
		& \mbox{ in } Q_{\rho_2}^{k,\varepsilon}, \  k\in\mathcal{I}^{1,g}_\varepsilon, \\
		\smallskip
		u_{\varepsilon} 
		& \mbox{ in } Q^{k,\varepsilon}_{\rho_2,1}, \  k\in\mathcal{I}^{1,g}_\varepsilon,\\
		\hat{u}_{k,\varepsilon} 
		& \mbox{ in } Q^{k,\varepsilon}\cap \Om, \  k\in\mathcal{I}^{1,b}_\varepsilon\cup \mathcal{I}^{2}_\varepsilon.
		\end{cases}
	\end{equation*}
Clearly $w_\e\in SBV^2(\Om)$; we now show that $w_\e$ converges to $u_\xi$ strongly in $L^1(\Om)$.

As a first step we show that $\sup_{\e}MS(w_{\varepsilon})<+\infty$. 
\smallskip

\noindent
From \eqref{GCubes} applied to the good cubes, we have that for every $\varepsilon >0$
	\begin{align*}
		MS\Bigg(w_{\varepsilon}, \bigcup_{k\in\mathcal{I}^{1,g}_\varepsilon} Q^{k,\varepsilon}\Bigg) = 
		\sum_{k\in\mathcal{I}^{1,g}_\varepsilon} \bigg(
		\int_{Q_{\rho_2}^{k,\varepsilon}}|D \tilde{u}_{k,\varepsilon}|^{2} dx +
		MS\big({u}_{\varepsilon}, Q^{k,\varepsilon}_{\rho_2,1}\big)\bigg)
		\leqslant (2+\eta) F_{\varepsilon}(u_{\varepsilon}).
	\end{align*}
For the bad cubes, by \eqref{uhatk} and \eqref{num_bc} we have, for $\varepsilon >0$, 
	\begin{align*}
	MS\Bigg(w_{\varepsilon}, \bigcup_{k\in\mathcal{I}^{1,b}_\varepsilon} Q^{k,\varepsilon}\Bigg) =
	\sum_{k\in\mathcal{I}^{1,b}_\varepsilon} MS\big(\hat{u}_{k,\varepsilon}, Q^{k,\varepsilon}\big)
		\leqslant C\sum_{k\in\mathcal{I}^{1,b}_\varepsilon} MS\big(u_{\varepsilon}, Q^{k,\varepsilon}_{\frac{1}{2},1}\big)
		\leqslant F_{\varepsilon}(u_{\varepsilon}),
	\end{align*}
where $C_{n}$ is the perimeter of $Q \setminus  P$ in $Q$. Similarly, for the boundary cubes we have 
\begin{align*}
	MS\Bigg(w_{\varepsilon}, \bigcup_{k\in\mathcal{I}^{2}_\varepsilon} \big(Q^{k,\varepsilon}\cap \Om\big)\Bigg) 
		&\leqslant \sum_{k\in\mathcal{I}^{2}_\varepsilon} MS\big(u_{\varepsilon}, Q^{k,\varepsilon}_{\frac{1}{2},1}\cap \Om\big)
		+ \mathcal{H}^{n-1}(\partial \Omega)
		\leqslant F_{\varepsilon}(u_{\varepsilon}) + C.
	\end{align*}
Eventually, since $F_{\varepsilon}(u_\e)\leqslant C$ for every $\e>0$, we get the desired uniform bound on $MS(w_{\varepsilon})$.

Moreover, since $\|u_{\varepsilon}\|_{L^{\infty}(\Om)}\leqslant \|u_{\xi}\|_{L^{\infty}(\Om)}$ the function $w_\e$ above can be constructed in a way such that $\|w_{\varepsilon}\|_{L^{\infty}(\Om)}\leqslant \|u_{\xi}\|_{L^{\infty}(\Om)}$. Hence, we can apply the compactness result \cite[Theorem 4.8]{AFP} to the sequence $(w_\e)$ to deduce the existence of $w \in SBV^2(\Omega)$ such that (up to a subsequence not relabelled) $w_\e$ converges to $w$ weakly$^*$ in $BV(\Om)$ and hence strongly in $L^1(\Om)$.
\smallskip

It remains to show that $w=u_\xi$. We observe that 
$$
w_\e = u_\e \ \textrm{ in } \ A_{\rho_2}^\e:= \bigcup_{k\in \mathcal{I}_\e^1\cup  \mathcal{I}_\e^2} Q_{\rho_{2},1}^{k,\varepsilon} \cap \Omega,
$$
and $\chi_{A_{\rho_2}^\e} \rightharpoonup C(\rho_2)$ weakly$^{*}$ in $L^\infty(\Om)$, for some constant $0<C(\rho_2)<1$.  Moreover by assumption $u_{\varepsilon} \rightarrow u_\xi$ in the sense of Definition \ref{def:conv} so that there exists a sequence $(\tilde u_\e) \subset L^1(\Omega)$ such that $\tilde u_\e = u_\e$ in $\Omega \cap \e P$, and $\tilde u_\e$ converges to $u_\xi$ strongly in $L^1(\Omega)$. Hence, since $A_{\rho_2}^\e\subset \Omega \cap \e P$,
$$
0= \int_{A_{\rho_2}^\e}|\tilde u_\e - w_\e| dx = \int_{\Om}|\tilde u_\e - w_\e|\chi_{A_{\rho_2}^\e} dx \to C(\rho_2) \int_{\Om}|u_\xi - w| dx, 
$$
as $\e\to 0$. Since $C(\rho_2)>0$ then necessarily $w=u_\xi$.

\medskip

\noindent \textit{Step 5: Energy estimate for $w_\e$.} First of all, from \eqref{GCubes} and \eqref{BCubes} we have
	\begin{equation}\label{F_e}
		F_{\varepsilon}(u_{\varepsilon}) 
		\geqslant \bigg(1 - \dfrac{\eta}{1 + \eta}\bigg) \int_{\Omega}\phi_\e(x)
		|\nabla w_{\varepsilon}|^{2} dx,
	\end{equation}
where 
\begin{equation*}
\phi_\e(x):= 
\begin{cases}
0 & \quad \textrm{if } x\in Q^{k,\varepsilon}_{\frac{1}{2}}\cap \Om, \  k\in\mathcal{I}^{1,b}_\varepsilon\cup \mathcal{I}^{2}_\varepsilon,\\
1 & \quad \textrm{otherwise in } \Omega.
\end{cases}
\end{equation*}
Note that $\phi_\e \to 1$ in measure as $\e\to 0$, by \eqref{num_bc} and since the cardinality of $\mathcal{I}^{2}_\varepsilon$ is of order $\e^{1-n}$.
\medskip

By the previous step and by \cite[Theorem 4.7]{AFP} we have
\begin{equation*}
		\liminf_{\varepsilon \rightarrow 0}
		\mathcal{H}^{n-1}(S_{w_{\varepsilon}} \cap \Omega)
		\geqslant \mathcal{H}^{n-1}(S_{u_{\xi}} \cap \Omega) =0,
	\end{equation*}
and since $\nabla w_{\varepsilon}$ converges to $\xi$ weakly in $L^{2}(\Om)$, and $\phi_\e \to 1$ in measure, we conclude that
	\begin{equation}
		\liminf_{\varepsilon \rightarrow 0}
		\int_{\Omega}\phi_\e(x)|\nabla w_{\varepsilon}|^{2} dx 
		\geqslant \int_{\Omega}|D u_{\xi}|^{2} dx = \mathcal{L}^n(\Om) |\xi|^2.
		\label{LSC}
	\end{equation}
Eventually passing to the liminf in \eqref{F_e} and appealing to \eqref{LSC} gives
	\begin{equation*}
		\liminf_{\varepsilon \rightarrow 0} 
		F_{\varepsilon}(u_{\varepsilon})
		\geqslant \bigg(1 - \dfrac{\eta}{1 + \eta}\bigg) 
		\mathcal{L}^n(\Om) |\xi|^2.
	\end{equation*}
Finally, by letting $\eta \rightarrow0^+$ we deduce that
	\begin{equation*}
		\liminf_{\varepsilon \rightarrow 0} F_{\varepsilon}(u_{\varepsilon}) 
		\geqslant \mathcal{L}^{n}(\Omega)|\xi|^{2},
	\end{equation*}
which concludes the proof.
\end{proof}

We are now ready to prove Theorem \ref{id:f_infty}.

\begin{proof}[Proof of Theorem \ref{id:f_infty}] Lemma \ref{ILV1} gives $
		f^\infty_{\rm{hom}}(\xi) 
		\leqslant |\xi|^{2}$ for every $\xi\in \R^n$, hence it only remains to prove the opposite inequality. 
		By $\Gamma$-convergence we have that there exists a sequence $(u_\e)$ converging to $u_{\xi}$ in the sense of Definition \ref{def:conv} such that
$$
\mathcal{L}^{n}(\Omega) f^\infty_{\rm hom}(\xi) = F^\infty_{\rm hom}(u_\xi)=\lim_{\e\to 0} F_\e (u_\e),
$$
hence the desired inequality immediately follows from \eqref{FL4}.  		
\end{proof}

\bigskip

\subsection{Critical case: $\ell \in (0,+\infty)$.} 
We start by proving a simple result, which is essentially a corollary of Lemma \ref{PL1}. Then the main result of this section is Corollary \ref{coro:nq} which asserts that the homogenized  volume integrand $f_{\rm{hom}}^\ell$ is \emph{not} $2$-homogeneous, unlike the extreme cases $\ell=0$ and $\ell=+\infty$, and unlike the volume integrand of the functionals $F_\e$.

\begin{lemma}\label{l_quadr:hat}
Let $\ell\in (0,+\infty)$, and let $f_{\rm{hom}}^\ell$ be as in Theorem \ref{t:int-rep}. Then $f^\ell_{\rm{hom}}$ is $2$-homogeneous if and only if $f^\ell_{\rm{hom}}(\xi) = \hat f(\xi)$ for every $\xi \in \R^n$.
\end{lemma}

\begin{proof}
Assume that $f_{\rm{hom}}^\ell$ is $2$-homogeneous. 
Replacing $\xi$ by $\lambda\xi$ in \eqref{ILV1}, with $\lambda\neq0$, gives
			\begin{equation*}
				\lambda^{2}\hat{f}(\xi)
				\leqslant\lambda^{2}f_{\rm{hom}}^\ell(\xi)
				\leqslant \min\left\{\lambda^{2}|\xi|^{2}, \lambda^{2}\hat{f}(\xi) + C\ell \right\},
			\end{equation*}
		which can be rewritten as 
			\begin{equation*}
				\hat{f}(\xi)
				\leqslant f_{\rm{hom}}^\ell(\xi)
				\leqslant \min\left\{|\xi|^{2}, \hat{f}(\xi) + \dfrac{C\ell}{\lambda^{2}} \right\}.
			\end{equation*}
		By letting $|\lambda|\rightarrow +\infty$, we have
			\begin{equation*}
				 f_{\rm{hom}}^\ell(\xi) =  \hat{f}(\xi),
			\end{equation*}
		where we have used the obvious bound $\hat f(\xi) \leqslant |\xi|^2$. 

\end{proof}

In the following result we assume that $\beta_{\varepsilon} = \varepsilon$ for convenience. 
	
\begin{propo}\label{fneqfhat}
	For every $\xi \in \mathbb{R}^{n}\setminus\{0\}$, and every $\ell \in (0,+\infty)$, we have
		\begin{equation*}
		 f^\ell_{\rm{hom}}(\xi) \neq \hat{f}(\xi).
		\end{equation*}
	\end{propo}
	
	\begin{proof} 
	Clearly the statement reduces to proving that $\hat{f}(\xi) < f^\ell_{\rm{hom}}(\xi)$ for every $\xi \in \mathbb{R}^{n}\setminus \{0\}$.
	
		We first note that from the definition of $\hat{f}$ in \eqref{F} we have that, if $\xi \neq 0$,
			\begin{equation*}
				\hat{f}(\xi) 
				\leqslant \int_{Q_{\frac{1}{2},1}} |\xi|^{2} dx
				< \int_{Q} |\xi|^{2} dx
				= |\xi|^{2}.
			\end{equation*}
		Hence
			\begin{equation}
				\hat{f}(\xi)<|\xi|^{2} \ \ \textrm{for every } \xi \in \mathbb{R}^{n}\setminus \{0\}.
				\label{CR1}
			\end{equation}
		To prove the claim, it is enough to show that for every $\xi\neq0$
		and for every admissible sequence $u_{\varepsilon}$ which converges to 
		$u_{\xi}$ in the sense of Definition \ref{def:conv} we have
			\begin{equation}
				\hat{f}(\xi)
				< \limsup_{\varepsilon \rightarrow 0}F_{\varepsilon}(u_{\varepsilon},Q).
				\label{CR2}
			\end{equation}
		Indeed, if the statement \eqref{CR2} is proven, then we can choose $u_\e$ to be the recovery sequence 
		of $F_{\varepsilon}$ for $u_\xi$ and deduce, from the $\Gamma$-convergence of $F_\e$ to $F^\ell_{\rm{hom}}$ for 
		$0<\ell<+\infty$, that for 
			\begin{equation*}
				\hat{f}(\xi)
				< \limsup_{\varepsilon \rightarrow 0}F_{\varepsilon}(u_{\varepsilon},Q)
				= f^\ell_{\rm{hom}}(\xi).
			\end{equation*}
		We can assume $\sup_\e F_{\varepsilon}(u_{\varepsilon},Q) < +\infty$, otherwise there is nothing to prove. We can also assume that $\|u_\e\|_{L^\infty(Q)}\leqslant  \|u_\xi\|_{L^\infty(Q)}$.  

\medskip

Proceeding as in the proof of Theorem \ref{ILVII}, we fix $\vartheta>0$ independent of $\varepsilon$ and introduce a classification of the cubes of the form $Q^{k,\varepsilon} :=\varepsilon(Q + k)$, with $k \in \mathbb{Z}^{n}$, for the cubes well contained in $Q$, as follows. We call a cube $Q^{k,\varepsilon}$ undamaged if it satisfies 
			\begin{equation*}
				\mathcal{H}^{n-1}(S_{u^{\varepsilon}} \cap Q^{k,\varepsilon}) 
				\leqslant \vartheta\varepsilon^{n-1},
			\end{equation*} 
		and damaged otherwise, namely if 
			\begin{equation*}
				\mathcal{H}^{n-1}(S_{u^{\varepsilon}} \cap Q^{k,\varepsilon})  
				> \vartheta\varepsilon^{n-1}.
			\end{equation*} 
Let $\mathcal{I}^d_\e$ be the set of damaged or boundary cubes, and let $N_{d}(\varepsilon)$ denote its cardinality. 

\medskip

\noindent
Again, as in the proof of Theorem \ref{ILVII}, we can construct an improved sequence $(w^\e)\subset L^1(Q)$ such that $w_\e$ converges to $u_\xi$ weakly in $BV(Q)$, and
\begin{equation}\label{important:estimate}
F_{\varepsilon}(u_{\varepsilon}) 
		\geqslant \bigg(1 - \dfrac{\eta}{1 + \eta}\bigg) \int_{Q}\phi_\e(x)
		|\nabla w_{\varepsilon}|^{2} dx + \e N_d(\e) \vartheta \e^{n-1},
	\end{equation}
where 
\begin{equation}\label{ae}
\phi_\e(x):= 
\begin{cases}
0 & \quad \textrm{if } x\in Q^{k,\varepsilon}_{\frac12}\cap \Om, \  k\in\mathcal{I}^{d}_\varepsilon,\\
1 & \quad \textrm{otherwise in } \Omega.
\end{cases}
\end{equation}
Note that $\|\phi_\e-1\|_{L^1(Q)} \leqslant c N_d(\e)\e^n$.
\medskip

We consider two different cases.
\smallskip

\noindent
If the number of damaged cubes is small\ie
			\begin{equation*}
				\limsup_{\varepsilon \rightarrow 0}\varepsilon^{n}N_{d}(\varepsilon)=0,
			\end{equation*}
since $\phi_\e \to 1$ in measure and $\nabla w_\e \rightharpoonup u_\xi$ weakly in $L^2(Q)$, by taking the liminf in \eqref{important:estimate} we obtain directly that 
\begin{equation*}
\liminf_{\e\to 0} F_{\varepsilon}(u_{\varepsilon}) 
		\geqslant \bigg(1 - \dfrac{\eta}{1 + \eta}\bigg) |\xi|^2,
	\end{equation*}
which by \eqref{CR1} and by the arbitrariness of $\eta>0$ implies \eqref{CR2}.

\smallskip

\noindent
On the other hand, if the number of damaged cubes is large\ie
			\begin{equation*}
				\limsup_{\varepsilon \rightarrow 0}\varepsilon^{n}N_{d}(\varepsilon)=C>0, 
			\end{equation*}
then there exists an infinitesimal subsequence $(\varepsilon_{k})$ such that
			\begin{equation*}
				\lim_{k \rightarrow +\infty}\varepsilon_{k}^{n}N_{d}(\varepsilon_{k}) 
				=\limsup_{\varepsilon \rightarrow 0}\varepsilon^{n}N_{d}(\varepsilon) = C.
			\end{equation*}
From \eqref{important:estimate} we deduce that 
\begin{equation*}
F_{\varepsilon_k}(u_{\varepsilon_k}) 
		\geqslant \bigg(1 - \dfrac{\eta}{1 + \eta}\bigg) \widehat{F}_{\e_k}(w_{\varepsilon_k}) + \e_k N_d(\e_k) \vartheta \e_k^{n-1},
\end{equation*}
which, by the $\Gamma$-convergence of $\widehat{F}_\e$ to $\widehat{F}$, implies that			
\begin{align*}
\limsup_{\e\to 0} F_{\varepsilon}(u_{\varepsilon}) & \geqslant \limsup_{k\to \infty} F_{\varepsilon_k}(u_{\varepsilon_k})   \geqslant \liminf_{k\to \infty} F_{\varepsilon_k}(u_{\varepsilon_k})  \\
		&\geqslant \bigg(1 - \dfrac{\eta}{1 + \eta}\bigg) \liminf_{k\to \infty} \widehat{F}_{\e_k}(w_{\varepsilon_k}) +  \vartheta C 
		\geqslant  \bigg(1 - \dfrac{\eta}{1 + \eta}\bigg) \, \hat{f}(\xi) +  \vartheta C
\end{align*}			
and hence the claim  \eqref{CR2}, by the arbitrariness of $\eta>0$. This concludes the proof.
	\end{proof}

A direct consequence of the previous results is that $f^\ell_{\rm{hom}}$ is not $2$-homogeneous.

\begin{coro}[$f^\ell_{\rm{hom}}$ is not $2$-homogeneous]\label{coro:nq}
		Let $\ell \in (0,+\infty)$; then the function $f^\ell_{\rm{hom}}$ is not $2$-homogeneous.
	\end{coro}

	\begin{proof}
	The conclusion is a straightforward consequence of Lemma \ref{l_quadr:hat} and Proposition \ref{fneqfhat}.
	\end{proof}

We can also say something about the behaviour of $f^\ell_{\rm{hom}}$ close to zero.

\begin{coro}[$f^\ell_{\rm{hom}}=|\xi|^2$ close to zero]
		Let $\ell\in (0,\infty)$; then there exists a constant $\gamma_0>0$ such that $f^\ell_{\rm{hom}}(\xi) = |\xi|^2$ for every $|\xi|\leqslant \gamma_0$.
	\end{coro}

	\begin{proof}
	Let $\xi\in \R$, let $u_{\varepsilon}$ be a recovery sequence for $F_\e$, converging to $u_{\xi}$ in the sense of Definition \ref{def:conv}, and let $w_\e$ be the improved sequence satisfying
\begin{equation*}
F_{\varepsilon}(u_{\varepsilon}) 
		\geqslant \bigg(1 - \dfrac{\eta}{1 + \eta}\bigg) \int_{Q}\phi_\e(x)
		|\nabla w_{\varepsilon}|^{2} dx + \e N_d(\e) \vartheta \e^{n-1},
\end{equation*}
with $\phi_\e$ as in \eqref{ae}. Clearly 
\begin{equation}\label{important:estimate5}
|\xi|^2 \geqslant f_{\rm{hom}}^\ell(\xi) 
		\geqslant \bigg(1 - \dfrac{\eta}{1 + \eta}\bigg) \liminf_{\e\to 0}\int_{Q}\phi_\e(x)
		|\nabla w_{\varepsilon}|^{2} dx + \liminf_{\e\to 0}\e N_d(\e) \vartheta \e^{n-1}.
\end{equation}
Note that, if $\xi$ is small, then the number of damaged cubes for $w_\e$ has to be small. If not, then from \eqref{important:estimate5} and by the arbitrariness of $\eta>0$ we would deduce that 
\begin{equation*}
|\xi|^2 \geqslant f_{\rm{hom}}^\ell(\xi) 
		\geqslant \hat{f}(\xi) + C\vartheta > |\xi|^2,
\end{equation*}
where the last inequality follows by the smallness of $\xi$, since $C>0$, and this would lead to a contradiction. Hence there exists $\gamma_0>0$ such that whenever $|\xi|<\gamma_0$, the improved sequence $w_\e$ has 
only a negligible number of damaged cubes, in which case, from \eqref{important:estimate5}, we would have that 
\begin{equation*}
|\xi|^2 \geqslant f_{\rm{hom}}^\ell(\xi) 
		\geqslant \bigg(1 - \dfrac{\eta}{1 + \eta}\bigg) |\xi|^2,
\end{equation*}
and so the claim.
\end{proof}

\begin{rem}[Asymptotic behaviour of $F_{\rm{hom}}^\ell$ at infinity] From Lemma \ref{PL1} we know that for every $\xi\in \R^n$ and every $\ell\in (0,+\infty)$
\begin{equation}\label{e:repeat}
			\hat{f}(\xi) 
			\leqslant f^\ell_{\rm{hom}}(\xi) 
			\leqslant \min \{|\xi|^{2}, \hat{f}(\xi) 
			+ C \ell\},
		\end{equation}  
	where $C>0$ is a constant depending only on $P$ and on $\Omega$.
Note that for $|\xi|$ sufficiently large \eqref{e:repeat} simplifies to 
\begin{equation*}
			\hat{f}(\xi) 
			\leqslant f^\ell_{\rm{hom}}(\xi) 
			\leqslant \hat{f}(\xi) 
			+ C \ell,
		\end{equation*}
which implies that, for $|\xi|$ large,  $|f^\ell_{\rm{hom}}(\xi) - \hat{f}(\xi)|\leqslant C$, uniformly in $\xi$.
In particular, if we divide by $\hat{f}(\xi)$ the previous estimate and let $|\xi|\to +\infty$ we deduce the limit behaviour
$$
\lim_{|\xi|\to +\infty} \frac{f^\ell_{\rm{hom}}(\xi)}{\hat{f}(\xi)} = 1.
$$

\end{rem}

\medskip


\section*{Acknowledgments}
The authors wish to thank Gianni Dal Maso and Massimiliano Morini for stimulating discussions. L. Scardia acknowledges support by the EPSRC under the Grant EP/N035631/1 ``Dislocation patterns 
beyond optimality''.
%


\end{document}